\documentclass[12pt]{amsart} \usepackage{amssymb} \usepackage{mathabx} \usepackage{bm}

\newcommand{\wrlab}[1]{\label{#1}}

\newtheorem{thm}{THEOREM}[section]

\newtheorem{lem}[thm]{LEMMA} 
\newtheorem{cor}[thm]{COROLLARY} 
 
 \newtheorem{thm*}{THEOREM}[]

\newcommand{\secref}[1]{\S\ref{#1}}

\newcommand{\tref}[1]{Theorem~\ref{#1}}
\newcommand{\cref}[1]{Corollary~\ref{#1}}

\newcommand{\lref}[1]{Lemma~\ref{#1}}

\def\N{{\mathbb N}}  \def\Q{{\mathbb Q}}
\def\R{{\mathbb R}}

\def\scrd{{\mathcal D}} \def\scrf{{\mathcal F}} 
 \def\scrk{{\mathcal K}} 
 \def\scrp{{\mathcal P}} \def\scrq{{\mathcal Q}}
\def\scrr{{\mathcal R}}  
\def\scru{{\mathcal U}}  
  
 \def\scrh{{\mathcal H}} \def\scrc{{\mathcal C}}

\def\bfz{{\bf0}}

\font\tenolde=eufm10 at 10pt
\font\sevenolde=eufm7
\font\fiveolde=eufm5
\newfam\oldefam
\textfont\oldefam=\tenolde
\scriptfont\oldefam=\sevenolde
\scriptscriptfont\oldefam=\fiveolde

 \def\Lc{{\mathfrak c}} \def\Lg{{\mathfrak g}}

\def\Ls{{\mathfrak s}}

\def\id{{\rm id}}
\def\dim{\hbox{dim\,}}

\def\det{\hbox{det\,}}
\def\exp{\hbox{exp}}

\def\Stab{\hbox{Stab}}
\def\Fix{\hbox{Fix}}

\def\Tay{\hbox{Tay}}
\def\lin{\hbox{lin}}

\def\dom{\hbox{dom}}
\def\im{\hbox{im}}

\begin{document}
\bibliographystyle{alpha}


\title[Local freeness for $C^\infty$ actions]
{Local freeness in frame bundle prolongations of $C^\infty$ actions}

\keywords{prolongation, moving frame, dynamics}
\subjclass{57Sxx, 58A05, 58A20, 53A55}

\author{Scot Adams}
\address{School of Mathematics\\ University of Minnesota\\Minneapolis, MN 55455
\\ adams@math.umn.edu}

\date{May 24, 2017\qquad Printout date: \today}

\begin{abstract}
We prove a local freeness result for $C^\infty$ actions.
\end{abstract}

\maketitle

 

\section{Introduction\wrlab{sect-intro}}

Let a real Lie group have a $C^\infty$ action on a $C^\infty$ real manifold $M$.
Let~$G^\circ$ be the identity component of~$G$.
Assume the fixpoint set of~any nontrivial element of~$G^\circ$ has empty interior in~$M$.
Let $n:=\dim G$.
Assume $n\ge1$.
Let $F$ be the frame bundle of~$M$ of order $n-1$.
We prove (\tref{thm-high-stabs}): there exists a $G$-invariant dense open subset~$Q$ of~$F$
such that the $G$-action on $Q$ has discrete stabilizers.

\section{Global notation, conventions and observations\wrlab{sect-global}}

Let $\N:=\{1,2,3,\ldots\}$ and $\N_0:=\N\cup\{0\}$.
For any set~$S$, the {\bf identity function on $S$}, denoted $\id_S:S\to S$, is defined by $\id_S(s)=s$.

Let $f$ be a function.
We will denote the domain of $f$ by $\dom[f]$
and the image of $f$ by $\im[f]:=f(\dom[f])$.
For any set $T$, we will denote the $f$-preimage of $T$ by
$f^*(T):=\{x\in\dom[f]\,|\,f(x)\in T\}$.
For any~$S\subseteq\dom[f]$,
define $f|S:S\to\im[f]$ by $(f|S)(s)=f(s)$.
For any function $g$, define
$g\circ f:f^*(\dom[g])\to\im[g]$
by $(g\circ f)(x)=g(f(x))$.

Throughout this paper, by ``manifold'', we mean
``Hausdorff, second countable, finite dimensional $C^\infty$  real manifold without boundary'',
unless otherwise specified.
By ``vector space'', we mean ``real vector space'', unless otherwise specified.
By ``group'', we mean ``multiplicative group'', unless otherwise specified.
By ``Lie group'', we mean ``real Lie group'', unless otherwise specified.
By ``action'', we mean ``left action'', unless otherwise specified.
Throughout this paper, 
every finite dimensional vector space is given, without comment,
its standard topology and manifold structure.
A nonempty open subset of~a topological space acquires, without comment,
its relative topology, inherited from the ambient topological space.

For any group $G$, the identity element of $G$ is $1_G$.
For any~vector space~$V$, the zero element of~$V$ is $0_V$.
For all $n\in\N$, let $0_n:=0_{\R^n}$;
then $0_n=(0,\ldots,0)\in\R^n$.
For any set $S$, for any vector space $V$,
the zero map $\bfz_{S,V}:S\to V$ is defined by $\bfz_{S,V}(s)=0_V$.
For any topological space $X$,
by ``$\forall^\circ x\in X$'', we mean
``there exists a dense open subset of~$X$ such that, for all $x$ in that subset''.

Let $M$ be a manifold. Then the tangent bundle of $M$ is denoted~$TM$.
Also, for any~$x\in M$, the tangent space at~$x$ of $M$ is denoted $T_xM$.

Let $M$ and~$N$ be manifolds.
Let $f:M\to N$ be $C^\infty$.
Then the differential of~$f$ is denoted $df:TM\to TN$.
Also, for any $x\in M$, the differential at~$x$ of $f$
is denoted $(df)_x:=(df)|(T_xM):T_xM\to T_{f(x)}N$.
We say that {\bf$f$ has constant rank} if:
there exists $r\in\N_0$ such that, for all $x\in M$,
$\dim(\im[(df)_x])=r$.

Let $M$ be a manifold.
Let $d:=\dim M$.
Let $L\subseteq M$.
For any integer $p\in[0,d]$,
we say $L$ is a {\bf locally closed $p$-submanifold of~$M$}
if, for all~$x\in L$, there are a neighborhood $U$ in $M$ of $x$
and a~$C^\infty$~diffeomorphism $\phi:U\to\R^d$
such that $\phi(L\cap U)$ is a $p$-dimensional subspace of~$\R^d$.
We say $L$ is a {\bf locally closed submanifold of $M$}
if there is an integer $p\in[0,d]$ such that $L$ is
a locally closed $p$-submanifold of~$M$.

Any open subset of a manifold $M$ is a locally closed submanifold of~$M$.
Any closed subgroup of a Lie group~$G$ is a locally closed submanifold of $G$.
Any affine subspace of a finite dimensional vector space~$V$
is a locally closed submanifold of $V$.

Let $L$ be a locally closed submanifold of a manifold $M$.
Then $L$~acquires, without comment,
the unique manifold structure for which the inclusion map $\iota:L\to M$ is an immersion.
The image of $d\iota:TL\to TM$ is a locally closed subset of~$TM$, and will be denoted $TL$.
There is a slight technical difference between this image, denoted $TL$,
and the tangent bundle of $L$, also denoted $TL$.
We ask the careful reader to determine which $TL$ is meant from context;
they are typically identified.
For any $x\in L$, the image of $(d\iota)_x:T_xL\to T_xM$ is denoted $T_xL$.
Again, it is common to identify the two vector spaces $T_xL$.
If $L$ is open in $M$, then, for all $x\in L$,
we have $T_xL=T_xM$.

Let $G$ be a group acting on a set $X$.
For any $x\in X$, we denote the stabilizer in $G$ of $x$
by $\Stab_G(x):=\{g\in G\,|\,gx=x\}$.
For any $g\in G$, we denote the fixpoint set in $X$ of $g$
by  $\Fix_X(g):=\{x\in X\,|\,gx=x\}$.
The $G$-action on $X$ is {\bf effective} if,
for all~$g\in G\backslash\{1_G\}$,
we have $\Fix_X(g)\ne X$.
The $G$-action on $X$ is {\bf free} if,
for all $g\in G\backslash\{1_G\}$,
we have $\Fix_X(g)=\emptyset$.
The $G$-action on $X$ is effective iff
$\displaystyle{\bigcap_{x\in X}\,\Stab_G(x)=\{1_G\}}$.
The $G$-action on $X$ is free iff,
for all $x\in X$, we have $\Stab_G(x)=\{1_G\}$.

Let $G$ be a topological group acting on a set $X$.
The $G$-action on $X$ is {\bf locally free} if,
for all $x\in X$, we have: $\Stab_G(x)$ is discrete in $G$.

Let $G$ be a group acting on a topological space $X$.
We will say that the action is {\bf fixpoint rare} if,
for all $g\in G\backslash\{1_G\}$,
the interior in~$X$ of~$\Fix_X(g)$ is empty.
Fixpoint rare implies: effective on all nonempty invariant open sets.
A partial converse of~this is \lref{lem-fixpt-rare-lcgp}.

Let $G$ be a group acting on a manifold $M$.
Assume: for all $g\in G$, $x\mapsto gx:M\to M$ is $C^\infty$.
The $G$-action on $M$ induces a $G$-action on~$TM$.
For all $x\in M$, let
$\Stab'_G(x):=\{g\in G\,|\,\forall v\in T_xM, gv=v\}$.

Let $G$ be a Lie group acting on a manifold $M$.
Assume that the $G$-action on $M$ is $C^\infty$,
{\it i.e.}, that $(g,x)\mapsto gx:G\times M\to M$ is $C^\infty$.
Then both of the following functions are upper semi-continuous:
\begin{eqnarray*}
x\,\,\mapsto\,\,\dim(\Stab_G(x))&:&M\,\,\to\,\,\N_0\\
\hbox{and}\qquad x\,\,\mapsto\,\,\dim(\Stab'_G(x))&:&M\,\,\to\,\,\N_0.
\end{eqnarray*}

Let $M$ and $N$ be manifolds.
Let $\scru$ be the set of open subsets of $M$.
We define $\displaystyle{C_O^\infty(M,N):=\bigcup_{U\in\,\scru}\,C^\infty(U,N)}$.
We let $D_O^\infty(M,N)$ denote the set of all
maps $f\in C_O^\infty(M,N)$ such that:
${\bm(}\,\im[f]\hbox{ is open in }N\,{\bm)}$
and ${\bm(}\,f:\dom[f]\to\im[f]\hbox{ is a }C^\infty\hbox{ diffeomorphism}\,{\bm)}$.

Let $M$ be a manifold, $d:=\dim M$.
We will let $\scrc_M:=D_O^\infty(M,\R^d)$ denote the set of charts on $M$,
and will let $\scrr_M:=D_O^\infty(\R^d,M)$ denote the set of reverse charts on~$M$.
Let $\scrr_M^0:=\{\lambda\in\scrr_M\,|\,0_d\in\dom[\lambda]\}$.

Let $d\in\N$.
Let $\alpha_1,\ldots,\alpha_d\in\N_0$.
Let $\alpha:=(\alpha_1,\ldots,\alpha_d)\in\N_0^d$.
We set $|\alpha|:=\alpha_1+\cdots+\alpha_d$
and $\alpha!:=(\alpha_1!)\cdots(\alpha_d!)$.
For all $x_1,\ldots,x_d\in\R$,
for~$x:=(x_1,\ldots,x_d)\in\R^d$,
we set $x^\alpha:=x_1^{\alpha_1}\cdots x_d^{\alpha_d}\in\R$.
Let $E_1,\ldots,E_d$ be the standard framing of $\R^d$.
For all $\delta\in\{1,\ldots,d\}$,
for any finite dimensional vector space $V$,
the vector field $E_\delta$
corresponds to a differential operator $\partial_\delta:C_O^\infty(\R^d,V)\to C_O^\infty(\R^d,V)$.
For any finite dimensional vector space $V$,
we define $\partial^\alpha:=\partial_1^{\alpha_1}\circ\cdots\circ\partial_d^{\alpha_d}:
C_O^\infty(\R^d,V)\to C_O^\infty(\R^d,V)$.

Let $d\in\N$ and let $D:=\{1,\ldots,d\}$.
We define a partial ordering $\le$ on~$\N_0^d$ by:
{\bf(} $(\rho_1,\ldots,\rho_d)\le(\sigma_1,\ldots,\sigma_d)$ {\bf)} iff {\bf(} $\forall i\in D,\,\rho_i\le\sigma_i$ {\bf)}.
For all~$\rho,\sigma\in\N_0^d$,
if $\rho\le\sigma$, then we define
$\displaystyle{\left(\!\!\!\begin{array}{c}\sigma\\\rho\end{array}\!\!\!\right):=\frac{\sigma!}{[\rho!][(\sigma-\rho)!]}}$.
Then, for all~$x,y\in\R^d$, for all $\sigma\in\N_0^d$, for $R:=\{\rho\in\N_0^d\,|\,\rho\le\sigma\}$,
we have a vector binomial formula: $\displaystyle{(x+y)^\sigma=\sum_{\rho\in R}\left(\!\!\!\begin{array}{c}\sigma\\\rho\end{array}\!\!\!\right)x^\rho y^{\sigma-\rho}}$.

Let $M$ and $N$ be manifolds, let $d:=\dim M$ and let $k\in\N_0$.
Let $A:=\{\alpha\in\N_0^d~\hbox{s.t.}~|\alpha|\le k\}$.
Let $f,\phi\in C_O^\infty(M,N)$ and let $p\in M$.
By {\bf$f$~agrees with $\phi$ to order~$k$ at $p$},
written $f\sim\phi\,\,[[k,p]]$, we mean:
\begin{itemize}
\item[] {\bf(} $\,\,$ $p\in(\dom[f])\cap(\dom[\phi])$ $\,\,$ {\bf)} $\,\,$ and $\,\,$ {\bf(} $\,\,$ $f(p)=\phi(p)$ $\,\,$ {\bf)} $\,\,$ and
\item[] {\bf(} $\,\,$ there exist $\lambda\in\scrr_M^0$ and $\mu\in\scrc_N$ such that
\begin{itemize}
\item[] $\,\,$\quad {\bf[} $\lambda(0_d)=p$ {\bf]} \quad and \quad {\bf[} $f(p)\in\dom[\mu]$ {\bf]} \quad and
\item[] {\bf[} $\forall\alpha\in A$, \quad $(\partial^\alpha(\mu\circ f\circ\lambda))(0_d)\,=\,(\partial^\alpha(\mu\circ\phi\circ\lambda))(0_d)$ {\bf]} $\,\,$ {\bf)}.
\end{itemize}
\end{itemize}

Let $M$ and $N$ be manifolds.
Let $f,\phi\in C_O^\infty(M,N)$
and let $p\in M$.
Assume that $p\in(\dom[f])\cap(\dom[\phi])$.
Then we have both
\begin{itemize}
\item[]{\bf[} {\bf(} $f\sim\phi\,\,[[0,p]]$ {\bf)} $\,\Leftrightarrow\,$ {\bf(} $f(p)=\phi(p)$ {\bf)} {\bf]} \quad\quad and
\item[]{\bf[} {\bf(} $f\sim\phi\,\,[[1,p]]$ {\bf)} $\,\Leftrightarrow\,$ {\bf(} $(df)_p=(d\phi)_p$ {\bf)} {\bf]}.
\end{itemize}

Let $L$, $M$ and $N$ be manifolds.
Let $\alpha\in C_O^\infty(L,M)$.
Let $p\in\dom[\alpha]$.
Let $k\in\N_0$ and let $f,\phi\in C_O^\infty(M,N)$.
Assume $f\sim\phi\,\,[[k,\alpha(p)]]$.
Then, by the Chain Rule, $f\circ\alpha\sim\phi\circ\alpha\,\,[[k,p]]$.

Let $L$, $M$ and $N$ be manifolds.
Let $k\in\N_0$, $p\in L$, $f,\phi\in C_O^\infty(L,M)$.
Assume $f\sim\phi\,\,[[k,p]]$.
Let $\omega\in C_O^\infty(M,N)$.
Assume $f(p)\in\dom[\omega]$.
Then, by the Chain Rule, $\omega\circ f\sim\omega\circ\phi\,\,[[k,p]]$.

Let $M$ and $N$ be manifolds and $\scrf\subseteq C_O^\infty(M,N)$.
Let $p\in M$, $k\in\N_0$.
Assume: $\forall f\in\scrf$, $p\in\dom[f]$.
Then $\{(f,\phi)\in\scrf\times\scrf\,|\,f\sim\phi\,\,[[k,p]]\}$
is an equivalence relation on $\scrf$,
and, for any $f\in\scrf$,
the equivalence class in $\scrf$ of $f$ will be denoted
by~$J_p^k(f,\scrf):=\{\phi\in\scrf\,|\,f\sim\phi\,\,[[k,p]]\}$.

Let $V$ and $W$ be finite dimensional vector spaces.
For all $j\in\N_0$, $\scrp_{V,W}^j$ denotes the vector space of
all homogeneous polynomial functions $V\to W$ of degree $=j$.
Then $\scrp_{V,W}^1$ is the vector space of all~homogeneous linear transformations $V\to W$.
For all integers $i\ge0$, for all~integers $k\ge i$,
let $\displaystyle{\scrp_{V,W}^{i,k}:=\sum_{j=i}^k\,\scrp_{V,W}^j}$.
Then, for all $j\in\N_0$, $\scrp_{V,W}^{0,j}$~is the vector space of all
polynomial functions $V\to W$ of degree $\le j$.
Also, for all $j\in\N$, we have
$\scrp_{V,W}^{1,j}=\{P\in\scrp_{V,W}^{0,j}\,|\,P(0_V)=0_W\}$.

Let $V$ and $W$ be finite dimensional vector spaces.
Let $f\in C_O^\infty(V,W)$ and let $x\in\dom[f]$.
For all $j\in\N_0$, by $\Tay_x^j(f)$, we will denote 
the order $j$ Taylor approximation of $f$ at $x$,
{\it i.e.}, the unique $P\in\scrp_{V,W}^{0,j}$
such that $P\sim f\,\,[[j,x]]$.
Also, $\lin_x^f:=[\Tay_x^1(f)]-[(\Tay_x^1(f))(0_V)]\in\scrp_{V,W}^1$
denotes the homogeneous linear part of the Taylor series of $f$ at $x$.

Let $d\in\N$.
Let $V$ and $W$ be $d$-dimensional vector spaces.
Let $f\in C_O^\infty(V,W)$.
Let $x\in\dom[f]$.
According to the Inverse Function Theorem:
{\bf(} there exists a neighborhood $U$ in $\dom[f]$ of $x$ such that $f|U\in D_O^\infty(V,W)$ {\bf)}
$\Leftrightarrow$ {\bf(} $\lin_x^f:V\to W$ is invertible {\bf)}.

For all $d\in\N$,
let $\scrh_d:=\{\eta\in D_O^\infty(\R^d,\R^d)\,|\,0_d\in\dom[\eta],\eta(0_d)=0_d\}$.

Let $d\in\N$, $E:=\R^d$, $j\in\N_0$.
Let $\det:\scrp_{E,E}^1\to\R$ be the determinant.
Let $\scrq^+:=\{P\in\scrp_{E,E}^{1,j}\,|\,\det(\lin_z^P)\ne0\}$.
Then $\scrq^+$~is a dense open subset of $\scrp_{E,E}^{1,j}$.
Let $z:=0_E=0_d$.
For all $\eta\in\scrh_d$, let $J_d^j\eta:=J_z^j(\eta,\scrh_d)$.
Let $H_d^j:=\{J_d^j\eta\,|\,\eta\in\scrh_d\}$.
Then the function $J_d^j\eta\mapsto\Tay_z^j(\eta):H_d^j\to\scrq^+$ is a bijection.
We give $H_d^j$ the unique manifold structure making this bijection into a $C^\infty$ diffeomorphism, as follows.
Let $m^+:=\dim(\scrp_{E,E}^{1,j})$.
Let $\scrk^+$ be the set of vector space isomorphisms $\scrp_{E,E}^{1,j}\to\R^{m^+}$.
For all~$\kappa\in\scrk^+$, we define $C_\kappa^+:H_d^j\to\R^{m^+}$ by:
$$\forall\eta\in\scrh_d,\qquad C_\kappa^+(J_d^j\eta)\,\,=\,\,\kappa(\Tay_z^j(\eta)).$$
Then $\{C_\kappa^+\,|\,\kappa\in\scrk^+\}$ is an atlas on $H_d^j$, making $H_d^j$ into a manifold.
Also, $H_d^j$ is a group under the multiplication induced by composition, {\it i.e.},
under the multiplication given by:
$$\forall\eta,\theta\in\scrh_d,\qquad(J_d^j\eta)(J_d^j\theta)\,\,=\,\,J_d^j(\eta\circ\theta).$$
With this manifold and group structure, $H_d^j$ is a Lie group.

Let $M$ be a manifold and let $d:=\dim M$.
Let $z:=0_d$.
Let $j\in\N_0$.
For all~$\lambda\in\scrr_M^0$, let $J_M^j\lambda:=J_z^j(\lambda,\scrr_M^0)$.
Then the {\bf$j$th order frame bundle of~$M$} is
$F^jM:=\{J_M^j\lambda\,|\,\lambda\in\scrr_M^0\}$.
Define $\pi_M^j:F^jM\to M$ by:
$$\forall\lambda\in\scrr_M^0,\qquad\pi_M^j(J_M^j\lambda)\,\,=\,\,\lambda(z).$$

Let $d\in\N$ and let $z:=0_d$.
Let $M$ and $N$ be $d$-dimensional manifolds.
Let $f\in D_O^\infty(M,N)$.
Define $\zeta:\scrr_M^0\to M$ by $\zeta(\lambda)=\lambda(z)$.
Let $j\in\N_0$.
We then define $F^jf:(\pi_M^j)^*(\dom[f])\to(\pi_N^j)^*(\im[f])$ by:
$$\forall\lambda\in\zeta^*(\dom[f]),\qquad(F^jf)(J_M^j\lambda)\,\,=\,\,J_N^j(f\circ\lambda).$$

Let $d\in\N$, $j\in\N_0$.
Let $L$, $M$ and $N$ be $d$-dimensional manifolds.
With the definitions above, $F^j$ has the following functoriality property:
$\forall f\in D_O^\infty(L,M)$, $\forall g\in D_O^\infty(M,N)$,
we have $F^j(g\circ f)=(F^jg)\circ(F^jf)$.

Let $M$ be a manifold.
Let $j\in\N_0$.
We develop a manifold structure on $F^jM$:
Let $d:=\dim M$, $E:=\R^d$, $z:=0_E=0_d$.
Let $\det:\scrp_{E,E}^1\to\R$ be the determinant.
Let $\scrq:=\{P\in\scrp_{E,E}^{0,j}\,|\,\det(\lin_z^P)\ne0\}$.
Then $\scrq$~is a dense open subset of $\scrp_{E,E}^{0,j}$
and $J_E^j\lambda\mapsto\Tay_z^j(\lambda):F^jE\to\scrq$ is a~bijection.
Let $m:=\dim(\scrp_{E,E}^{0,j})$.
Let $\scrk$ be the set of vector space isomorphisms $\scrp_{E,E}^{0,j}\to\R^m$.
For all $\kappa\in\scrk$, define $C_\kappa:F^jE\to\R^m$ by:
$$\forall\nu\in\scrr_E^0,\qquad C_\kappa(J_E^j\nu)\,\,=\,\,\kappa(\Tay_z^j(\nu)).$$
For all $\kappa\in\scrk$, $\lambda\in\scrc_M$,
let $C_\kappa^\lambda:=C_\kappa\circ(F^j\lambda):(\pi_M^j)^*(\dom[\lambda])\to\R^m$.
Then $\{C_\kappa^\lambda\,|\,\kappa\in\scrk,\,\lambda\in\scrc_M\}$
is an atlas on $F^jM$, making $F^jM$ into a manifold.
Define a right action on $F^jM$ by $H_d^j$ by the rule:
$$\forall\lambda\in\scrr_M^0,\,\,\,\forall\eta\in\scrh_d,\qquad(J_M^j\lambda)(J_d^j\eta)\,\,=\,\,J_M^j(\lambda\circ\eta).$$
Then $\pi_M^j:F^jM\to M$ is a $C^\infty$~principal $H_d^j$-bundle.

Let $d\in\N$ and
let $M$ and $N$ both be $d$-dimensional manifolds.
Let $f\in D_O^\infty(M,N)$.
Let $j\in\N_0$.
Then $F^jf$ is a local isomorphism of~$C^\infty$~principal~$H_d^j$-bundles,
by which we mean all of the following:
\begin{itemize}
\item[$\bullet$]$F^jf\in D_O^\infty(F^jM,F^jN)$,
\item[$\bullet$]$\forall q\in\dom[F^jf],\,\,\qquad\pi_N^j((F^jf)(q))=\pi_M^j(q)$,
\item[$\bullet$]$\dom[F^jf]$ is invariant under the right action on $F^jM$ by $H_d^j$
\item[and $\bullet$]$\forall q\in\dom[F^jf],\,\,\forall h\in H_d^j,\qquad(F^jf)(qh)=[(F^jf)(q)]h$.
\end{itemize}

\section{Fibers of constant rank maps\wrlab{sect-fibers-const-rk}}

\begin{lem}\wrlab{lem-const-rank-consequence}
Let $L$ and $M$ be manifolds.
Let $f:L\to M$ be $C^\infty$.
Assume $f$ has constant rank.
Let $q\in f(L)$ and $A:=f^*(\{q\})$.
Then:
\begin{itemize}
\item[(i)]$A$ is a closed subset of $L$,
\item[(ii)]$A$ is a locally closed submanifold of $L$ \qquad and
\item[(iii)] for all $a\in A$, we have $T_aA=\ker[(df)_a]$.
\end{itemize}
\end{lem}

\begin{proof}
Result (i) follows from continuity of $f$.
Results (ii) and (iii) follow from the Constant Rank Theorem.
\end{proof}

\section{Orbits of $C^\infty$ actions \wrlab{sect-orbits-smooth-acts}}

Let $G$ be a Lie group acting on a manifold $M$.
We assume that the~$G$-action on $M$ is $C^\infty$.
For all $p\in M$, let $G_p:=\Stab_G(p)$,
and define $\overline{p}:G\to M$ by $\overline{p}(g)=gp$.

\begin{lem}\wrlab{lem-orbit-map-deriv}
Let $p\in M$ and let $a\in G$.
Then $aG_p$ is a closed subset of~$G$
and a locally closed submanifold of $G$.
Also, $T_a(aG_p)=\ker[(d\,\overline{p})_a]$.
\end{lem}

\begin{proof}
Let $G$ act on $G$ by left translation
so that the action of $g\in G$ on~$g'\in G$ yields $gg'\in G$.
This action induces an action of $G$ on~$TG$.
The map $\overline{p}:G\to M$ is $G$-equivariant.
Then $d\,\overline{p}:TG\to TM$ is also $G$-equivariant.
By transitivity of the $G$-action on $G$,
we see that $\overline{p}$ has constant rank.
Let $q:=\overline{p}(a)$.
Then $aG_p=\overline{p}^*(\{q\})$.
The result is then a consequence of \lref{lem-const-rank-consequence}
(with $L$ replaced by $G$ and $f$ by $\overline{p}$).
\end{proof}

\begin{cor}\wrlab{cor-orbit-map-deriv}
Let $p\in M$.
Then $G_p$ is a closed subset of $G$ and a locally closed submanifold of~$G$.
Also, for all $a\in G_p$, we  have $T_aG_p=\ker[(d\,\overline{p})_a]$.
\end{cor}

\begin{proof}
As {\bf(}$\forall a\in G_p,aG_p=G_p${\bf)},
the result follows from \lref{lem-orbit-map-deriv}.
\end{proof}

\section{Infinitesmal stabilizers in $C^\infty$ actions \wrlab{sect-inftsml-stabs}}

Let $G$ be a Lie group acting on a manifold $M$.
We assume that the~$G$-action on $M$ is $C^\infty$.
For all $p\in M$, let $G_p:=\Stab_G(p)$,
and define $\overline{p}:G\to M$ by $\overline{p}(g)=gp$.
Let $\Lg:=T_{1_G}G$.
For all $X\in\Lg$, for all~$p\in M$, let $Xp:=(d\,\overline{p})_{1_G}(X)\in T_pM$.
For all $p\in M$, let $\Lg_p:=T_{1_G}G_p$.

\begin{lem}\wrlab{lem-orbit-map-deriv-redux}
Let $p\in M$.
Then $\Lg_p=\ker[X\mapsto Xp:\Lg\to T_pM]$.
\end{lem}

\begin{proof}
Let $a:=1_G$.
Then, for all $X\in \Lg$, we have $Xp=(d\,\overline{p})_a(X)$,
so, since $\Lg_p=T_aG_p$, the result follows from \cref{cor-orbit-map-deriv}.
\end{proof}

\begin{cor}\wrlab{cor-inject-Lie-algebra}
Let $p\in M$.
Assume that $G_p$ is discrete.
Then $X\mapsto Xp:\Lg\to T_pM$ is injective.
\end{cor}

\begin{proof}
Since $G_p$ is discrete, we see that $\Lg_p=\{0_\Lg\}$.
Then, by \lref{lem-orbit-map-deriv-redux},
$X\mapsto Xp:\Lg\to T_pM$ has kernel $\{0_\Lg\}$,
and is therefore injective.
\end{proof}

\section{Fixpoint rare for connected locally compact groups\wrlab{sect-fixpt-rare-lcgp}}

\begin{lem}\wrlab{lem-fixpt-rare-lcgp}
Let $G$ be a connected locally compact Hausdorff topological group
acting on a Hausdorff topological space $X$.
Assume, for all $g\in G$, that the map $x\mapsto gx:X\to X$ is continuous.
Assume, for all $x\in X$, that the map $g\mapsto gx:G\to X$ is continuous.
Assume, for every nonempty $G$-invariant open subset $V$ of $X$,
that the $G$-action on~$V$ is effective.
Then the $G$-action on $X$ is fixpoint rare.
\end{lem}

\begin{proof}
Let $g_0\in G\backslash\{1_G\}$ be given, and let $F_0:=\Fix_X(g_0)$.
We wish to show that $F_0$ has empty interior in $X$.
Let $V_0$ be the interior in $X$ of~$F_0$.
Assume that $V_0\ne\emptyset$.
We aim for a contradiction.

Let $V_1:=GV_0$.
As $V_1$ is a nonempty $G$-invariant open subset of~$X$,
by~hypothesis, the $G$-action on $V_1$ is effective.
Then $\Fix_{V_1}(g_0)\subsetneq V_1$.
Choose $x_1\in V_1$ such that $x_1\notin\Fix_{V_1}(g_0)$.
Then we have $g_0x_1\ne x_1$.
Let $H:=\Stab_G(x_1)$.
Then $g_0\notin H$.
Let $U_0:=G\backslash H$ be the complement in~$G$ of~$H$.
Then $U_0$ is an~open neighborhood in $G$ of $g_0$.
Let $U:=g_0^{-1}U_0$.
Then $U$ is an open neighborhood in $G$ of $1_G$.
By the Approximation Theorem in~\S4.6, p.~175 of \cite{montzip:transfgps},
choose a compact normal subgroup $K$ of $G$ such that $K\subseteq U$
and such that $G/K$ admits a Lie group structure compatible with its quotient topology.
Let $\overline{G}:=G/K$, and give $\overline{G}$
its quotient topology and compatible Lie group structure.
Then give the Lie group $\overline{G}$ its natural $C^\omega$~manifold structure.
Let $\pi:G\to\overline{G}$ be the~canonical homomorphism.
Let $\overline{g_0}:=\pi(g_0)$.

Since $x_1\in V_1=GV_0$,
choose $c\in G$ and $x_0\in V_0$ such that $x_1=cx_0$.
Let $A:=\{g\in G\,|\,gx_1\in V_0\}$.
Then $Ax_1\subseteq V_0$.
As $c^{-1}x_1=x_0\in V_0$, it follows that $c^{-1}\in A$, and so $A\ne\emptyset$.
Also, since $g\mapsto gx_1:G\to X$ is continuous,
we see that $A$ is open in $G$.
Let $\overline{A}:=\pi(A)$.
Since $\pi:G\to\overline{G}$ is an open mapping
and since $A$ is a nonempty open subset of $G$,
we see that $\overline{A}$ is a nonempty open subset of $\overline{G}$.
Let $\overline{H}:=\pi(H)$.
Then $\pi^*(\overline{H})=HK$.
Since $H$ is closed in $G$ and $K$ is compact,
it follows that $HK$ is closed in $G$.
Then $\overline{H}$ is closed in $\overline{G}$.
Let $Y:=\overline{G}/\,\overline{H}$.
Let $p:\overline{G}\to Y$ be the canonical map.
By the theory of homogeneous spaces of Lie groups,
give $Y$ the unique $C^\omega$~manifold structure under which $p:\overline{G}\to Y$ is a $C^\omega$ submersion.
Let $y_1:=p(1_{\overline{G}})$.
Then $\Stab_{\overline{G}}(y_1)=\overline{H}$.

{\it Claim 1:} $\forall y\in Y,\,\overline{g_0}y=y$.
{\it Proof of Claim 1:}
Let $W:=p(\overline{A})$.
Since $G$ is connected, $\overline{G}$~is connected, so $Y$ is connected.
Because $p:\overline{G}\to Y$ is open and because $\overline{A}$~is a nonempty open subset of $\overline{G}$,
we see that $W$~is a nonempty open subset of $Y$.
Then, as $y\mapsto\overline{g_0}y:Y\to Y$ is $C^\omega$,
it suffices to show, for all $w\in W$, that $\overline{g_0}w=w$.
Let $w\in W$ be given.
We wish to prove that $\overline{g_0}w=w$.

As $w\in W=p(\overline{A})=p(\pi(A))$,
choose $a\in A$ such that $w=p(\pi(a))$.
Let $g_1:=a^{-1}g_0a$.
Let $\overline{g_1}:=\pi(g_1)$ and $\overline{a}:=\pi(a)$.
Then $\overline{a}\,\overline{g_1}\,\overline{a}^{-1}=\overline{g_0}$.
Because $ax_1\in Ax_1\subseteq V_0\subseteq F_0=\Fix_X(g_0)$,
we see that $g_0ax_1=ax_1$, so $g_1x_1=a^{-1}g_0ax_1=a^{-1}ax_1=x_1$.
Then $g_1\in\Stab_G(x_1)=H$,
and it follows that
$\overline{g_1}=\pi(g_1)\in\pi(H)=\overline{H}=\Stab_{\overline{G}}(y_1)$.
Then $\overline{g_1}y_1=y_1$,
Moreover, we have $w=p(\pi(a))=p(\overline{a})=\overline{a}[p(1_{\overline{G}})]=\overline{a}y_1$.
Therefore, $\overline{g_0}w=\overline{a}\,\overline{g_1}\,\overline{a}^{-1}\,\overline{a}y_1
=\overline{a}\,\overline{g_1}y_1=\overline{a}y_1=w$.
{\it End of proof of Claim~1.}

By Claim 1, we have $\overline{g_0}y_1=y_1$, so $\overline{g_0}\in\Stab_{\overline{G}}(y_1)$.
It follows that $\pi(g_0)=\overline{g_0}\in\Stab_{\overline{G}}(y_1)=\overline{H}$.
Then $g_0\in\pi^*(\overline{H})=HK$.
Choose $h\in H$, $k\in K$ such that $g_0=hk$.
Then $h=g_0k^{-1}\in g_0K\subseteq g_0U=U_0=G\backslash H$, so $h\notin H$.
Then $h\in H$ and $h\notin H$.
Contradiction.
\end{proof}

In \lref{lem-fixpt-rare-lcgp}, we cannot drop the assumption that $G$ is connected,
even if we add the assumption that $X$ is connected:

\vskip.1in\noindent
{\it Example:}
Let $X:=\R$.
Let $\scrd$ denote the group of all $C^\infty$ diffeomorphisms of $X$.
Then $\scrd$~acts effectively on~$X$.
Choose $f\in\scrd$ such that $\Fix_X(f)=[-1,1]$.
For all~$q\in\Q$, let $T_q\in\scrd$ be defined by $T_q(x)=q+x$.
Let $G$ be the subgroup of $\scrd$ generated by $\{f\}\cup\{T_q\,|\,q\in\Q\}$.
Since $\scrd$~acts effectively on $X$,
it follows that the $G$-action on $X$ is also effective.
Give to $G$ the discrete topology and manifold structure.
Then the $G$-action on $X$ is~$C^\infty$.
For all $x\in X$, since $\Q+x=\{T_q(x)\,|\,q\in\Q\}\subseteq Gx$,
we see that $Gx$ is dense in~$X$.
Then, for any nonempty $G$-invariant subset~$V$ of $X$,
we see that $V$ is dense in~$X$, and so, because the $G$-action on $X$ is effective,
it follows that the $G$-action on~$V$ is~effective as well.
On the other hand, $f\in G\backslash\{1_G\}$ and $\Fix_X(f)=[-1,1]$, so the $G$-action on~$X$ is not fixpoint rare.
{\it End of example.}

\section{Induced maps on polynomial spaces\wrlab{sect-induced-polynomials}}

Let $j,d\in\N$.
Let $E:=\R^d$.
Let $V:=\scrp_{E,E}^{0,j}$.
Let $z:=0_E=0_d$.
Define $\zeta:V\to E$ by $\zeta(P)=P(z)$.
For all $\omega\in C_O^\infty(E,E)$,
define $\omega_*:\zeta^*(\dom[\omega])\to V$ by $\omega_*(P)=\Tay_z^j(\omega\circ P)$;
then $\omega_*\in C_O^\infty(V,V)$.

\begin{lem}\wrlab{lem-drop-j-derivs}
Let $f\in C_O^\infty(E,E)$.
Let $i\in\N_0$.
Let $Q\in V$,
and assume that $f\sim\bfz_{E,E}\,\,[[i+j,Q(z)]]$.
Then $f_*\sim\bfz_{V,V}\,\,[[i,Q]]$.
\end{lem}

\begin{proof}
Let $V_0:=\scrp_{E,\R}^{0,j}$.
Let $(P,y)\mapsto P\odot y:V_0\times E\to V$
be the bilinear map defined by $(P\odot y)(w)=[P(w)]y$.
Let $A:=\{\alpha\in\N_0^d~\hbox{s.t.}~|\alpha|\le j\}$.
Let $U:=\dom[f]$ and let $p:=Q(z)$.
For all $\alpha\in A$, let $H_\alpha:E\to V_0$
be defined by $\displaystyle{(H_\alpha(x))(w)=\frac{(w-x)^\alpha}{\alpha!}}$,
and let $\eta_\alpha:U\to V$
be defined by~$\eta_\alpha(x)=[H_\alpha(x)]\odot[(\partial^\alpha f)(x)]$.
Since $f\sim\bfz_{E,E}\,\,[[i+j,p]]$,
it follows, for all $\alpha\in A$, that $\partial^\alpha f\sim\bfz_{E,E}\,\,[[i,p]]$.
So, by the Product Rule for $\odot$,
we conclude, for all $\alpha\in A$, that $\eta_\alpha\sim\bfz_{E,V}\,\,[[i,p]]$.

Define $\tau:U\to V$ by $\tau(x)=\Tay_x^j(f)$.
For all $x\in U$, for all $w\in E$,
\begin{eqnarray*}
(\tau(x))(w)&=&\sum_{\alpha\in A}\,\,\left[\frac{(w-x)^\alpha}{\alpha!}\right]\,\left[(\partial^\alpha f)(x)\right]\\
&=&\sum_{\alpha\in A}\,\,[(H_\alpha(x))(w)]\,[(\partial^\alpha f)(x)]\\
&=&\sum_{\alpha\in A}\,\,\bigg([H_\alpha(x)]\odot[(\partial^\alpha f)(x)]\bigg)(w)\\
&=&\sum_{\alpha\in A}\,\,(\eta_\alpha(x))(w).
\end{eqnarray*}
Then $\displaystyle{\tau=\sum_{\alpha\in A}\,\eta_\alpha}$.
Then $\tau\sim\bfz_{E,V}\,\,[[i,p]]$.
So, since $\zeta(Q)=Q(z)=p$,
it follows that $\tau\circ\zeta\sim\bfz_{E,V}\circ\zeta\,\,[[i,Q]]$.
Let $\mu_0:=\tau\circ\zeta:\zeta^*(U)\to V$.
Let $\nu_0:=\bfz_{V,V}:V\to V$.
Then we have $\nu_0=\bfz_{E,V}\circ\zeta$.
It follows that $\mu_0\sim\nu_0\,\,[[i,Q]]$.
Define $\mu:\zeta^*(U)\to V\times V$ and $\nu:V\to V\times V$ by
$$\mu(P)=(\mu_0(P),P)\qquad\hbox{and}\qquad\nu(P)=(\nu_0(P),P).$$
Then $\mu\sim\nu\,\,[[i,Q]]$.
Define $\Omega:V\times V\to V$ by $\Omega(P,\Pi)=\Tay_z^j(P\circ\Pi)$.
Then $\Omega\circ\mu\sim\Omega\circ\nu\,\,[[i,Q]]$.
It therefore suffices to show both that $f_*=\Omega\circ\mu$ and that $\bfz_{V,V}=\Omega\circ\nu$.

For all $P\in V$, we have $\bfz_{V,V}(P)=0_V=\bfz_{E,E}$, and so
\begin{eqnarray*}
(\Omega\circ\nu)(P)&=&\Omega(\nu(P))\,\,=\,\,\Omega(\nu_0(P),P)\,\,=\,\,\Omega(\bfz_{V,V}(P),P)\\
&=&\Omega(\bfz_{E,E},P)\,\,=\,\,\Tay_z^j(\bfz_{E,E}\circ P)\\
&=&\Tay_z^j(\bfz_{E,E})\,\,=\,\,\bfz_{E,E}\,\,=\,\,\bfz_{V,V}(P).
\end{eqnarray*}
Then $\Omega\circ\nu=\bfz_{V,V}$.
It remains to show that $f_*=\Omega\circ\mu$.

Because $\im[\mu]\subseteq V\times V=\dom[\Omega]$,
we get $\dom[\Omega\circ\mu]=\dom[\mu]$.
Then $\dom[\Omega\circ\mu]=\zeta^*(U)$.
Also, $\dom[f_*]=\zeta^*(\dom[f])=\zeta^*(U)$.
Let $P\in\zeta^*(U)$ be given.
We wish to prove that $f_*(P)=(\Omega\circ\mu)(P)$.

Let $u:=P(z)$.
Then $u=\zeta(P)$.
Then $u\in\zeta(\zeta^*(U))\subseteq U=\dom[f]$.
Let $F:=\Tay_u^j(f)$.
Then $F\sim f\,\,[[j,u]]$,
so, since $P(z)=u$, we see that $F\circ P\sim f\circ P\,\,[[j,z]]$.
Then $\Tay_z^j(F\circ P)=\Tay_z^j(f\circ P)$.
Also, we have $\mu_0(P)=(\tau\circ\zeta)(P)=\tau(\zeta(P))=\tau(u)=\Tay_u^j(f)=F$.
Then
$(\Omega\circ\mu)(P)\,\,=\,\,\Omega(\mu(P))\,\,=\,\,\Omega(\mu_0(P),P)\,\,=\,\,\Omega(F,P)\,\,=\,\,\Tay_z^j(F\circ P)$.
Then $f_*(P)=\Tay_z^j(f\circ P)=\Tay_z^j(F\circ P)=(\Omega\circ\mu)(P)$, as desired.
\end{proof}

\begin{lem}\wrlab{lem-add-drop-j-derivs}
Let $f\in C_O^\infty(E,E)$.
Let $i\in\N_0$.
Let $I:=id_E\in V$.
Assume that $z\in\dom[f]$.
Then:
$${\bm(}\quad f\sim\bfz_{E,E}\,\,[[i+j,z]]\quad{\bm)}\quad\quad\Leftrightarrow\quad\quad{\bm(}\quad f_*\sim\bfz_{V,V}\,\,[[i,I]]\quad{\bm)}.$$
\end{lem}

\begin{proof}
By \lref{lem-drop-j-derivs} (with $Q$ replaced by $I$),
we get $\Rightarrow$.
Assume that $f_*\sim\bfz_{V,V}\,\,[[i,I]]$.
We wish to prove that $f\sim\bfz_{E,E}\,\,[[i+j,z]]$.

Let $F:=\Tay_z^{i+j}(f)$.
Then $f\sim F\,\,[[i+j,z]]$,
so it suffices to show that $F=\bfz_{E,E}$.
Let $Y:=\scrp_{E,E}^{0,i+j}$.
Then $F\in Y$ and $0_Y=\bfz_{E,E}$.
It therefore suffices to prove that $F=0_Y$.

Let $\phi:=F-f$.
Then $\phi_*=F_*-f_*$.
Since $F\sim f\,\,[[i+j,z]]$,
we conclude that $\phi\sim\bfz_{E,E}\,\,[[i+j,z]]$.
Then, by \lref{lem-drop-j-derivs} (with $Q$ replaced by $I$ and $f$ by $\phi$),
we see that $\phi_*\sim\bfz_{V,V}\,\,[[i,I]]$.
So, since $f_*\sim\bfz_{V,V}\,\,[[i,I]]$
and since $F_*=\phi_*+f_*$,
we get $F_*\sim\bfz_{V,V}\,\,[[i,I]]$.

For all $w\in E$, define $T_w:E\to E$ by $T_w(x)=w+x$;
then $T_w\in V$.
Also, $T_z=I$.
Define $T_\bullet:E\to V$ by~$T_\bullet(w)=T_w$.
Then $T_\bullet(z)=I$.
So, because $F_*\sim\bfz_{V,V}\,\,[[i,I]]$,
we get $F_*\circ T_\bullet\sim\bfz_{V,V}\circ T_\bullet\,\,[[i,z]]$.
Let $\Lambda:=F_*\circ T_\bullet$.
Then, since $\bfz_{V,V}\circ T_\bullet=\bfz_{E,V}$,
we get $\Lambda\sim\bfz_{E,V}\,\,[[i,z]]$.

Let $S:=\{\sigma\in\N_0^d~\hbox{s.t.}~|\sigma|\le i+j\}$.
Let $D:=\{1,\ldots,d\}$.
Let $C:=\R^{S\times D}$ denote the vector space of all functions $S\times D\to\R$.
For all $\gamma\in C$, for all~$\sigma\in S$, for all~$\delta\in D$,
we simplify notation by defining $\gamma_\delta^\sigma:=\gamma(\sigma,\delta)$.
Let $\varepsilon_1,\ldots\varepsilon_d$ be the standard basis of $E=\R^d$.
For all~$\gamma\in C$, let $P_\gamma:E\to E$ be defined by
$\displaystyle{P_\gamma(x)=\sum_{\delta\in D}\sum_{\sigma\in S}\,\gamma_\delta^\sigma x^\sigma\varepsilon_\delta}$;
then $P_\gamma\in Y$.
The map $\gamma\mapsto P_\gamma:C\to Y$ is a vector space isomorphism.
Let $c\in C$ satisfy $P_c=F$.
We wish to show that $c=0_C$.

For all $\sigma\in S$,
let $R_\sigma:=\{\rho\in\N_0^d~\hbox{s.t.}~\rho\le\sigma\}$.
For all $w,x\in E$,
\begin{eqnarray*}
(F\circ T_w)(x)&=&F(T_w(x))\,\,=\,\,F(w+x)\,\,=\,\,P_c(w+x)\\
&=&\sum_{\delta\in D}\,\,\sum_{\sigma\in S}\,\,c_\delta^\sigma(w+x)^\sigma\varepsilon_\delta\\
&=&\sum_{\delta\in D}\,\,\sum_{\sigma\in S}\,\,c_\delta^\sigma\sum_{\rho\in R_\sigma}\left(\!\!\!\begin{array}{c}\sigma\\\rho\end{array}\!\!\!\right)w^\rho x^{\sigma-\rho}\varepsilon_\delta.
\end{eqnarray*}
For all $w\in E$,
$\Lambda(w)=(F_*\circ T_\bullet)(w)=F_*(T_\bullet(w))=F_*(T_w)$.
For all~$\sigma\in S$,
let $R'_\sigma:=\{\rho\in R_\sigma~\hbox{s.t.}~|\sigma-\rho|\le j\}$.
Then, for all $w,x\in E$,
\begin{eqnarray*}
(\Lambda(w))(x)&=&(F_*(T_w))(x)\,\,=\,\,(\Tay_z^j(F\circ T_w))(x)\\
&=&\sum_{\delta\in D}\,\,\sum_{\sigma\in S}\,\,c_\delta^\sigma\sum_{\rho\in R'_\sigma}
\left(\!\!\!\begin{array}{c}\sigma\\\rho\end{array}\!\!\!\right)w^\rho x^{\sigma-\rho}\varepsilon_\delta.
\end{eqnarray*}

Let $\sigma_0\in S$ and $\delta_0\in D$ be given.
We wish to show that $c_{\delta_0}^{\sigma_0}=0$.

Let $A:=\{\alpha\in\N_0^d~\hbox{s.t.}~|\alpha|\le i\}$.
Let $B:=\{\beta\in\N_0^d~\hbox{s.t.}~|\beta|\le j\}$.
Then $A+B=\{\sigma\in\N_0^d~\hbox{s.t.}~|\sigma|\le i+j\}=S$.
So, since $\sigma_0\in S=A+B$,
choose $\alpha_0\in A$ and $\beta_0\in B$ such that $\sigma_0=\alpha_0+\beta_0$.
Then $\alpha_0\le\alpha_0+\beta_0=\sigma_0$, so $\alpha_0\in R_{\sigma_0}$.
Also, we have $|\sigma_0-\alpha_0|=|\beta_0|\le j$.
Then $\alpha_0\in R'_{\sigma_0}$.
Let $S':=\{\sigma\in S~\hbox{s.t.}~\alpha_0\in R'_\sigma\}$.
Then $\sigma_0\in S'$.
Let $\displaystyle{\xi:=\frac{(\partial^{\alpha_0}\Lambda)(z)}{\alpha_0!}\in V}$.
Then, for all $x\in E$, we have
\begin{eqnarray*}
\xi(x)\quad&=&\quad\sum_{\delta\in D}\,\,\sum_{\sigma\in S'}\,\,c_\delta^\sigma\left(\!\!\!\begin{array}{c}\sigma\\\alpha_0\end{array}\!\!\!\right)x^{\sigma-\alpha_0}\varepsilon_\delta\\
&=&\quad\sum_{\delta\in D}\,\,\,\sum_{\beta\in S'-\alpha_0}\,\,\,c_\delta^{\alpha_0+\beta}\left(\!\!\!\begin{array}{c}\alpha_0+\beta\\\alpha_0\end{array}\!\!\!\right)x^\beta\varepsilon_\delta.
\end{eqnarray*}
Since $\alpha_0\in A$, we have $|\alpha_0|\le i$.
So, as $\Lambda\sim\bfz_{E,V}\,\,[[i,z]]$,
we get $(\partial^{\alpha_0}\Lambda)(z)=0_V$.
Then $\xi=0_V=\bfz_{E,E}$.
Since $\alpha_0+\beta_0=\sigma_0\in S'$, we get $\beta_0\in S'-\alpha_0$.
Let $\displaystyle{u:=\frac{(\partial^{\beta_0}\xi)(z)}{\beta_0!}\in E}$.
Then
$$u\,\,=\,\,\sum_{\delta\in D}\,\,c_\delta^{\alpha_0+\beta_0}\left(\!\!\!\begin{array}{c}\alpha_0+\beta_0\\\alpha_0\end{array}\!\!\!\right)\varepsilon_\delta
\,\,=\,\,\sum_{\delta\in D}\,\,c_\delta^{\sigma_0}\left(\!\!\!\begin{array}{c}\sigma_0\\\alpha_0\end{array}\!\!\!\right)\varepsilon_\delta.$$
As $\xi=\bfz_{E,E}$, we see that $(\partial^{\beta_0}\xi)(z)=0_E$.
Then $u=0_E=0_d$.
Let $\pi:E\to\R$ denote projection onto the $\delta_0$ coordinate,
defined by $\pi(x_1,\ldots,x_d)=x_{\delta_0}$.
Then $c_{\delta_0}^{\sigma_0}\left(\!\!\!\begin{array}{c}\sigma_0\\\alpha_0\end{array}\!\!\!\right)=\pi(u)=\pi(0_d)=0$.
So, since~$\left(\!\!\!\begin{array}{c}\sigma_0\\\alpha_0\end{array}\!\!\!\right)\ne0$,
we conclude that $c_{\delta_0}^{\sigma_0}=0$, as desired.
\end{proof}

\section{Induced maps on frame bundles\wrlab{sect-induced-frame}}

Let $d\in\N$ and let $M$ and $N$ both be $d$-dimensional manifolds.

\begin{lem}\wrlab{lem-lose-j-derivs-from-ell}
Let $f,\phi\in D_O^\infty(M,N)$,
$i,j\in\N_0$, $q\in F^jM$, $p:=\pi_M^j(q)$.
Then:\qquad
${\bm(}\,f\sim\phi\,\,[[i+j,p]]\,{\bm)}\quad\Leftrightarrow\quad{\bm(}\,F^jf\sim F^j\phi\,\,[[i,q]]\,{\bm)}$.
\end{lem}

\begin{proof}
If $j=0$, then, by identifying $F^jf$ with $f$ and $F^j\phi$ with $\phi$, the result follows.
We therefore assume that $j\in\N$.
We also assume
$${\bm(}\,p\in\dom[f]\,{\bm)}\quad\hbox{and}\quad{\bm(}\,p\in\dom[\phi]\,{\bm)}\quad\hbox{and}\quad{\bm(}\,f(p)=\phi(p)\,{\bm)};$$
otherwise, both ${\bm(}\,f\sim\phi\,\,[[i+j,p]]\,{\bm)}$
and ${\bm(}\,F^jf\sim F^j\phi\,\,[[i,q]]\,{\bm)}$ are false.

Let $h:=\phi^{-1}\circ f\in D_O^\infty(M,M)$.
Let $\iota:=\id_M:M\to M$.
Then $h(p)=p=\iota(p)$ and $F^jh=(F^j\phi)^{-1}\circ(F^jf)$.
It suffices to prove:
$${\bm(}\,h\sim\iota\,[[i+j,p]]\,{\bm)}\quad\Leftrightarrow\quad{\bm(}\,F^jh\sim F^j\iota\,[[i,q]]\,{\bm)}.$$

Let $E:=\R^d$, $z:=0_d$, $I:=\id_E:E\to E$, $\sigma:=J_E^jI\in F^jE$.
Then $\pi_E^j(\sigma)=\pi_E^j(J_E^jI)=I(z)=z$.
Choose $\lambda\in\scrr_M^0$ such that $q=J_M^j\lambda$.
Then $(F^j\lambda)(\sigma)=(F^j\lambda)(J_E^jI)=J_M^j(\lambda\circ I)=J_M^j\lambda=q$.
Also, we have $p=\pi_M^j(q)=\pi_M^j(J_M^j\lambda)=\lambda(z)$.
Let $\psi:=\lambda^{-1}\circ h\circ\lambda\in D_O^\infty(E,E)$.
Then $\psi(z)\!=\!z\!=\!I(z)$
and $F^j\psi\!=\!(F^j\lambda)^{-1}\circ(F^jh)\circ(F^j\lambda)$.
It suffices to prove:
$${\bm(}\,\psi\sim I\,[[i+j,z]]\,{\bf)}\quad\Leftrightarrow\quad
{\bf(}\,F^j\psi\sim F^jI\,[[i,\sigma]]\,{\bf)}.$$

Let $V:=\scrp_{E,E}^{0,j}$.
Let $\det:\scrp_{E,E}^1\to\R$ denote the determinant function.
Let $V^\times:=\{P\in V\,|\,\det(\lin_z^P)\ne0\}$.
Define $T:F^jE\to V^\times$ by~$T(J_E^j\lambda)=\Tay_z^j(\lambda)$.
Then $T:F^jE\to V^\times$ is a $C^\infty$ diffeomorphism.

Let $U:=\dom[\psi]$.
Then $U$ is an open neighborhood in $E$ of $z$.
Let $\chi:=\psi-I:U\to E$.
Define $\zeta:V\to E$ by $\zeta(P)=P(z)$.
For all~$\omega\in C_O^\infty(E,E)$,
define $\omega_*:\zeta^*(\dom[\omega])\to V$ by $\omega_*(P)=\Tay_z^j(\omega\circ P)$;
then $\omega_*\in C_O^\infty(V,V)$.
Since $\chi=\psi-I$, we get $\chi_*=\psi_*-I_*$.

{\it Claim 1:} Let $\omega\in D_O^\infty(E,E)$.
Then $\omega_*\circ T=T\circ(F^j\omega)$.
{\it Proof of~Claim~1:}
Because we have $\im[F^j\omega]\subseteq F^jE=\dom[T]$,
we see that $\dom[T\circ(F^j\omega)]=\dom[F^j\omega]$.
For all $\lambda\in\scrr_E^0$,
\begin{eqnarray*}
(\zeta\circ T)(J_E^j\lambda)&=&\zeta(T(J_E^j\lambda))\,\,=\,\,\zeta(\Tay_z^j(\lambda))\\
&=&(\Tay_z^j(\lambda))(z)\,\,=\,\,\lambda(z)\,\,=\,\,\pi_E^j(J_E^j\lambda).
\end{eqnarray*}
Then $\zeta\circ T=\pi_E^j$.
Let $Q:=(\pi_E^j)^*(\dom[\omega])$.
Then $Q\subseteq\dom[\pi_E^j]$.
Also,
\begin{eqnarray*}
\dom[\omega_*\circ T]&=&T^*(\dom[\omega_*])\,\,=\,\,T^*(\zeta^*(\dom[\omega]))\\
&=&(\pi_E^j)^*(\dom[\omega])\,\,=\,\,Q.
\end{eqnarray*}
Also, $\dom[T\circ(F^j\omega)]=\dom[F^j\omega]=(\pi_E^j)^*(\dom[\omega])=Q$.
Let $q\in Q$ be given.
We wish to show that $\omega_*(T(q))=T((F^j\omega)(q))$.

Since $q\in Q\subseteq\dom[\pi_E^j]=F^jE$,
choose $\mu\in\scrr_E^0$ such that $q=J_E^j\mu$.
Let $P:=\Tay_z^j(\mu)$.
Then $T(q)=T(J_E^j\mu)=\Tay_z^j(\mu)=P$.
Then $\omega_*(T(q))=\omega_*(P)=\Tay_z^j(\omega\circ P)$.
Also, $\pi_E^j(q)=\pi_E^j(J_E^j\mu)=\mu(z)$.
Then $\mu(z)=\pi_E^j(q)\in\pi_E^j(Q)=\pi_E^j((\pi_E^j)^*(\dom[\omega]))\subseteq\dom[\omega]$.

We have $(F^j\omega)(q)=(F^j\omega)(J_E^j\mu)=J_E^j(\omega\circ\mu)$.
Moreover, we have $T(J_E^j(\omega\circ\mu))=\Tay_z^j(\omega\circ\mu)$.
As $P=\Tay_z^j(\mu)$,
we get $P\sim\mu\,\,[[j,z]]$.
Then $\omega\circ P\sim\omega\circ\mu\,\,[[j,z]]$,
so $\Tay_z^j(\omega\circ P)=\Tay_z^j(\omega\circ\mu)$.
Then
\begin{eqnarray*}
\omega_*(T(q))&=&\Tay_z^j(\omega\circ P)\,\,=\,\,\Tay_z^j(\omega\circ\mu)\\
&=&T(J_E^j(\omega\circ\mu))\,\,=\,\,T((F^j\omega)(q)),
\end{eqnarray*}
as desired.
{\it End of proof of Claim 1.}

By Claim 1, $\psi_*\circ T=T\circ(F^j\psi)$ and $I_*\circ T=T\circ(F^jI)$.
Also, we have $T(\sigma)=T(J_E^jI)=\Tay_z^j(I)=I$.
Then
$${\bm(}\,\psi_*\,\sim\,I_*\,\,[[i,I]]\,{\bm)}
\quad\Leftrightarrow\quad
{\bm(}\,F^j\psi\sim F^jI\,\,[[i,\sigma]]\,{\bm)}.$$

By \lref{lem-add-drop-j-derivs},
{\bf(} $\chi\sim\bfz_{E,E}\,\,[[i+j,z]]$ {\bf)} $\Leftrightarrow$
{\bf(} $\chi_*\sim\bfz_{V,V}\,[[i,I]]$ {\bf)}.
So, since $\chi=\psi-I$ and since $\chi_*=\psi_*-I_*$, we get
$${\bm(}\,\psi\sim I\,\,[[i+j,z]]\,{\bm)}\quad\Leftrightarrow\quad
{\bm(}\,\psi_*\sim I_*\,[[i,I]]\,{\bm)}.$$

Putting this together, we have shown:
\begin{eqnarray*}
{\bm(}\,\psi\sim I\,\,[[i+j,z]]\,{\bm)}&\Leftrightarrow&{\bm(}\,\psi_*\sim I_*\,[[i,I]]\,{\bm)}\\
&\Leftrightarrow&{\bm(}\,F^j\psi\sim F^jI\,\,[[i,\sigma]]\,{\bm),}
\end{eqnarray*}
as desired.
\end{proof}

The $i=0$ special case of $\Leftarrow$
of \lref{lem-lose-j-derivs-from-ell}
is often useful, and it admits a simple proof,
so we present it separately:

\begin{lem}\wrlab{lem-stab-above-highstab-below}
Let $f,\phi\in D_O^\infty(M,N)$.
Let $j\in\N_0$.
Let $q\in F^jM$.
Assume that $q\in(\dom[F^jf])\cap(\dom[F^j\phi])$
and that $(F^jf)(q)=(F^j\phi)(q)$.
Let $p:=\pi_M^j(q)$.
Then $f\sim\phi\,\,[[j,p]]$.
\end{lem}

\begin{proof}
Since $q\in F^jM$, choose $\lambda\in\scrr_M^0$ such that $q=J_M^j\lambda$.
Then $(F^jf)(q)=J_N^j(f\circ\lambda)$ and $(F^j\phi)(q)=J_N^j(\phi\circ\lambda)$
and $\pi_M^j(q)=\lambda(0_d)$.

We have $J_N^j(f\circ\lambda)=(F^jf)(q)=(F^j\phi)(q)=J_N^j(\phi\circ\lambda)$.
It follows that $f\circ\lambda\sim\phi\circ\lambda\,\,[[j,0_d]]$.
Since $p=\pi_M^j(q)=\lambda(0_d)$, we get $\lambda^{-1}(p)=0_d$.
Then $f\circ\lambda\circ\lambda^{-1}\sim\phi\circ\lambda\circ\lambda^{-1}\,\,[[j,p]]$.
Then $f\sim\phi\,\,[[j,p]]$.
\end{proof}

\section{Iterated frame bundles contain frame bundles\wrlab{sect-iter-frame}}

Let $X$ be a manifold.
Let $d:=\dim X$, $E:=\R^d$, $z:=0_d$, $j\in\N_0$ and $m:=\dim(\scrp_{E,E}^{0,j})$.
Let $\kappa:\scrp_{E,E}^{0,j}\to\R^m$ be a fixed vector space isomorphism.
Recall (from \secref{sect-global}): $C_\kappa:F^jE\to\R^m$ is defined by
$$\forall\nu\in\scrr_E^0,\qquad C_\kappa(J_E^j\nu)\,\,=\,\,\kappa(\Tay_z^j(\nu)).$$
Then $C_\kappa\in\scrc_{F^jE}$.
Let $I:=\id_E$.
Let $q_0:=J_E^jI\in F^jE$.
Then we have $\pi_E^j(q_0)=\pi_E^j(J_E^jI)=I(z)=z$.
Let the function $B_\kappa:F^jE\to\R^m$ be defined by~$B_\kappa(q)=[C_\kappa(q)]-[C_\kappa(q_0)]$;
then $B_\kappa\in\scrc_{F^jE}$ and $B_\kappa(q_0)=0_m$.
Let $R_\kappa:=B_\kappa^{-1}$; then $R_\kappa\in\scrr_{F^jE}^0$ and $R_\kappa(0_m)=q_0$.
For all~$\lambda\in\scrr_X^0$,
$$R_\kappa(0_m)\,\,=\,\,q_0\,\,\in\,\,(\pi_E^j)^*(z)\,\,\subseteq\,\,(\pi_E^j)^*(\dom[\lambda])\,\,=\,\,\dom[F^j\lambda],$$
and we define $R_\kappa^\lambda:=(F^j\lambda)\circ R_\kappa\in\scrr_{F^jX}^0$.
Let $i\in\N_0$.

\begin{lem}\wrlab{lem-Rkappa-basic-fact}
For all $\lambda,\mu\in\scrr_X^0$, we have:
$${\bm(}\,\,\lambda\sim\mu\,\,[[i+j,z]]\,\,{\bm)}\,\,\Leftrightarrow\,\,{\bm(}\,\,R_\kappa^\lambda\sim R_\kappa^\mu\,\,[[i,0_m]]\,\,{\bm)}.$$
\end{lem}

\begin{proof}
For all $\lambda,\mu\in\scrr_X^0$, because of the definitions above,
we have both $R_\kappa^\lambda=(F^j\lambda)\circ R_\kappa$ and $R_\kappa^\mu=(F^j\mu)\circ R_\kappa$;
so, since $R_\kappa(0_m)=q_0$,
$${\bm(}\,F^j\lambda\sim F^j\mu\,\,[[i,q_0]]\,{\bm)}\,\,\Leftrightarrow\,\,
{\bm(}\,\,R_\kappa^\lambda\sim R_\kappa^\mu\,\,[[i,0_m]]\,\,{\bm)}.$$

By \lref{lem-lose-j-derivs-from-ell} (with $M$ replaced by $E$, $N$ by $X$, $q$ by $q_0$ and $p$ by~$z$),
we conclude that: for all $\lambda,\mu\in\scrr_X^0$,
$${\bm(}\,\,\lambda\sim\mu\,\,[[i+j,z]]\,\,{\bm)}\,\,\Leftrightarrow\,\,{\bm(}\,\,F^j\lambda\sim F^j\mu\,\,[[i,q_0]]\,\,{\bm)}.$$
Putting this together, for all $\lambda,\mu\in\scrr_X^0$,
\begin{eqnarray*}
{\bm(}\,\,\lambda\sim\mu\,\,[[i+j,z]]\,\,{\bm)}&\Leftrightarrow&{\bm(}\,\,F^j\lambda\sim F^j\mu\,\,[[i,q_0]]\,\,{\bm)}\\
&\Leftrightarrow&{\bm(}\,\,R_\kappa^\lambda\sim R_\kappa^\mu\,\,[[i,0_m]]\,\,{\bm)},
\end{eqnarray*}
as desired.
\end{proof}

\begin{lem}\wrlab{lem-natural-map}
There exists a unique function $\Phi:F^{i+j}X\to F^i(F^jX)$
satisfying the condition: $\forall\lambda\in\scrr_X^0$, $\Phi(J_X^{i+j}\lambda)=J_{F^jX}^iR_\kappa^\lambda$.
\end{lem}

\begin{proof}
This follows from $\Rightarrow$ of \lref{lem-Rkappa-basic-fact}.
\end{proof}

\begin{lem}\wrlab{lem-inj-of-natural-map}
Let the function $\Phi:F^{i+j}X\to F^i(F^jX)$ satisfy the condition:
$\forall\lambda\in\scrr_X^0$, $\Phi(J_X^{i+j}\lambda)=J_{F^jX}^iR_\kappa^\lambda$.
Then $\Phi$ is injective.
\end{lem}

\begin{proof}
This follows from $\Leftarrow$ of \lref{lem-Rkappa-basic-fact}.
\end{proof}

\section{Naturality of inclusion in iterated frame bundles\wrlab{sect-nat-incl}}

Let $d\in\N$.
Let $i,j\in\N_0$.
Let $E:=\R^d$.
Let $m:=\dim(\scrp_{E,E}^{0,j})$,
and let $\kappa:\scrp_{E,E}^{0,j}\to\R^m$ be a fixed vector space isomorphism.
Define $R_\kappa$ as in~\secref{sect-iter-frame}.
For any $d$-dimensional manifold $X$,
for any $\lambda\in\scrr_X^0$,
define $R_\kappa^\lambda\in\scrr_{F^jX}^0$ as in~\secref{sect-iter-frame}.
Following \lref{lem-natural-map},
for every $d$-dimensional manifold $X$,
let $\Phi_{X,\kappa}^{i,j}:F^{i+j}X\to F^i(F^jX)$ be the unique function satisfying the condition:
$\forall\lambda\in\scrr_X^0$, $\Phi_{X,\kappa}^{i,j}(J_X^{i+j}\lambda)=J_{F^jX}^iR_\kappa^\lambda$.

With $i$, $j$ and $\kappa$ fixed, we show that $\Phi_{X,\kappa}^{i,j}$ is natural in $X$:

\begin{lem}\wrlab{lem-naturality}
Let $M$ and $N$ both be $d$-dimensional manifolds,
and let $f\in D_O^\infty(M,N)$.
Then $(F^i(F^jf))\circ\Phi_{M,\kappa}^{i,j}=\Phi_{N,\kappa}^{i,j}\circ(F^{i+j}f)$.
\end{lem}

\begin{proof}
Let $M':=F^jM$, $N':=F^jN$, $f':=F^jf\in D_O^\infty(M',N')$.
Let $S:=\Phi_{M,\kappa}^{i,j}$, $T:=\Phi_{N,\kappa}^{i,j}$.
We wish to prove: $(F^if')\circ S=T\circ(F^{i+j}f)$.

Let $\Pi:=\pi_M^j\circ\pi_{M'}^i:F^iM'\to M$ and $\pi:=\pi_M^{i+j}:F^{i+j}M\to M$.
Let $I:=\id_E:E\to E$ and $q_0:=J_E^jI\in F^jE$.
For all $\lambda\in\scrr_M^0$,
$(F^j\lambda)(q_0)=(F^j\lambda)(J_E^jI)=J_M^j(\lambda\circ I)=J_M^j\lambda$.
Then, for all $\lambda\in\scrr_M^0$,
\begin{eqnarray*}
(\Pi\circ S)(J_M^{i+j}\lambda)
&=&(\pi_M^j\circ\pi_{M'}^i\circ\Phi_{M,\kappa}^{i,j})(J_M^{i+j}\lambda)\\
&=&(\pi_M^j\circ\pi_{M'}^i)(J_{M'}^iR_\kappa^\lambda)
\,\,=\,\,\pi_M^j(R_\kappa^\lambda(0_m))\\
&=&\pi_M^j({\bm(}(F^j\lambda)\circ R_\kappa{\bm)}{\bm(}0_m{\bm)})
\,\,=\,\,\pi_M^j((F^j\lambda)(q_0))\\
&=&\pi_M^j(J_M^j\lambda)
\,\,=\,\,\lambda(0_d)
\,\,=\,\,\pi(J_M^{i+j}\lambda).
\end{eqnarray*}
Then $\Pi\circ S=\pi$.
Also, $\dom[f']=\dom[F^jf]=(\pi_M^j)^*(\dom[f])$.
Then
\begin{eqnarray*}
\Pi^*(\dom[f])&=&(\pi_{M'}^i)^*((\pi_M^j)^*(\dom[f]))\\
&=&(\pi_{M'}^i)^*(\dom[f'])\,\,=\,\,\dom[F^if'].
\end{eqnarray*}

Because we have $\im[F^{i+j}f]\subseteq F^{i+j}N=\dom[T]$,
it follows that $\dom[T\circ(F^{i+j}f)]=\dom[F^{i+j}f]$.
Let $Q:=\pi^*(\dom[f])$.
Then
\begin{eqnarray*}
\dom[(F^if')\circ S]&=&S^*(\dom[F^if'])\,\,=\,\,S^*(\Pi^*(\dom[f]))\\
&=&\pi^*(\dom[f])\,\,=\,\,Q.
\end{eqnarray*}
Also, $\dom[T\circ(F^{i+j}f)]=\dom[F^{i+j}f]=\pi^*(\dom[f])=Q$.
Let $q\in Q$ be~given.
We wish to prove that
$(F^if')(S(q))=T((F^{i+j}f)(q))$.

We have $q\in Q=\pi^*(\dom[f])\subseteq\dom[\pi]=F^{i+j}M$,
so choose $\mu\in\scrr_M^0$ such that $q=J_M^{i+j}\mu$.
Then $S(q)=\Phi_{M,\kappa}^{i,j}(J_M^{i+j}\mu)=J_{M'}^iR_\kappa^\mu$.
Also, $\mu(0_d)=\pi(J_M^{i+j}\mu)=\pi(q)\in\pi(Q)=\pi(\pi^*(\dom[f]))\subseteq\dom[f]$.

Let $\phi:=f\circ\mu\in\scrr_N^0$.
By functoriality of $F^j$, $F^j\phi=(F^jf)\circ(F^j\mu)$.
We have $R_\kappa^\mu=(F^j\mu)\circ R_\kappa$
and $R_\kappa^\phi=(F^j\phi)\circ R_\kappa$.
Then
$$f'\,\circ\,R_\kappa^\mu\,\,=\,\,(F^jf)\,\circ\,(F^j\mu)\,\circ\,R_\kappa\,\,=\,\,(F^j\phi)\,\circ\,R_\kappa\,\,=\,\,R_\kappa^\phi.$$
Then
$(F^if')(S(q))=(F^if')(J_{M'}^iR_\kappa^\mu)=J_{N'}^i(f'\circ R_\kappa^\mu)=J_{N'}^iR_\kappa^\phi$.

We have $(F^{i+j}f)(q)=(F^{i+j}f)(J_M^{i+j}\mu)=J_N^{i+j}(f\circ\mu)=J_N^{i+j}\phi$.
Then 
$T((F^{i+j}f)(q))\!=\!T(J_N^{i+j}\phi)\!=\!\Phi_{N,\kappa}^{i,j}(J_N^{i+j}\phi)\!=\!J_{N'}^iR_\kappa^\phi\!=\!(F^if')(S(q))$.
\end{proof}

\section{Loss of dimension in stabilizers\wrlab{sect-decay-moving-stabs}}

Let a Lie group $G$ act on a manifold~$M$.
Assume that the action is~$C^\infty$.
The $G$-action on $M$ induces a $G$-action on $TM$.
For all $p\in M$, let $G_p:=\Stab_G(p)$
and $G'_p:=\Stab'_G(p)=\{g\in G\,|\,\forall v\in T_pM,\,gv=v\}$.
Let $\Lg:=T_{1_G}G$.
For all $p\in M$, let $\Lg_p:=T_{1_G}G_p$,
and let $\Lg'_p:=T_{1_G}G'_p$.
Let $G^\circ$ denote the identity component of $G$.

\begin{lem}\wrlab{lem-decay-moving-stabs}
Assume the $G^\circ$-action on $M$ is fixpoint rare.
Let $k\in\N$.
Assume: $\forall p\in M$, $\dim\Lg_p=k$.
Then: $\forall^\circ p\in M$, $\dim\Lg'_p<k$.
\end{lem}

Let $\Lg_\bullet:=p\mapsto \Lg_p$
denote the stabilizer map from $M$ to the manifold of $k$-dimensional subspaces of $\Lg$.
The basic theme of the proof below is:
Since each nontrivial element of $G^\circ$ has interior-free fixpoint set,
it follows that: $\forall^\circ p\in M$, the differential $(d(\Lg_\bullet))_p$ is nonzero.
Morever, at such a point $p$, we can show that $\Lg'_p\subsetneq\Lg_p$.
Consequently, $\forall^\circ p\in M$, $\dim\Lg'_p\le(\dim\Lg_p)-1=k-1<k$.
In words: ``By fixpoint rarity, the stabilizer map cannot be constant on a nonempty open set.
Moreover, wherever the stabilizer map is `on the move', it is strictly larger than the first order stabilizer.
Consequently, generically, the first order stabilizer is strictly smaller than the stabilizer.''
Details follow.

\begin{proof}
Define $S:=\{p\in M\,|\,\dim\Lg'_p<k\}$.
By upper semi-continuity of~$p\mapsto\dim\Lg'_p:M\to\N_0$,
we see that $S$ is open in $M$.
It therefore suffices to~show that $S$ is dense in $M$.
Let a nonempty open subset $M'$ of $M$ be given.
We wish to prove that $M'\cap S\ne\emptyset$.

For all $g\in G$, define $\overline{g}:M\to M$ by~$\overline{g}(p)=gp$.
For all $g\in G$, for all~$v\in TM$, we have $gv=(d\,\overline{g})(v)$.
For all $p\in M$, define $\overline{p}:G\to M$ by $\overline{p}(g)=gp$.
For all $x\in TG$, for all $p\in M$, let $xp:=(d\,\overline{p})(x)$.
By~\lref{lem-orbit-map-deriv-redux},
for all $p\in M$, the kernel of $X\mapsto Xp:\Lg\to T_pM$ is $\Lg_p$.

For all $p\in M$, let $0_p:=0_{T_pM}$.
For all $g\in G$, let $0_g:=0_{T_gG}$.

Let $e:=\exp:\Lg\to G^\circ$ be the Lie theoretic exponential map.
Choose an~open neighborhood $\Lg_*$ in $\Lg$ of $0_\Lg$
and an open neighborhood $G_*$ in~$G^\circ$ of~$1_G$
such that $e(\Lg_*)=G_*$ and such that $e|\Lg_*:\Lg_*\to G_*$ is a~$C^\infty$~diffeomorphism.
Let $\Lg_*^\times:=\Lg_*\backslash\{0_\Lg\}$.
Let $G_*^\times:=G_*\backslash\{1_G\}$.
Then $e(\Lg_*^\times)=G_*^\times$.
Since $G_*\subseteq G^\circ$, we conclude that
$G_*^\times\subseteq G^\circ\backslash\{1_G\}$.

As $M'\ne\emptyset$, choose $u\in M'$.
Choose a vector subspace $\Lc$ of $\Lg$
such that both $\Lc\cap\Lg_u=\{0_\Lg\}$ and $\Lc+\Lg_u=\Lg$.
Then $(\dim\Lc)+(\dim\Lg_u)=\dim\Lg$.
Let $\ell:=\dim\Lc$ and let $n:=\dim\Lg$.
Then, as $k=\dim\Lg_u$, we get $\ell+k=n$.
Choose $C_1,\ldots,C_\ell\in\Lc$ such that
$\{C_1,\ldots,C_\ell\}$ is a basis of $\Lc$.
Since $\Lc\cap\Lg_u$ is the kernel of
$C\mapsto Cu:\Lc\to T_uM$
and since $\Lc\cap\Lg_u=\{0_\Lg\}$,
it follows that $C\mapsto Cu:\Lc\to T_uM$ is injective.
So, because $C_1,\ldots,C_\ell$ are linearly independent in $\Lc$,
we conclude that $C_1u,\ldots,C_\ell u$ are linearly independent in~$T_uM$.
Choose an open neighborhood $M_1$ in $M'$ of $u$ such that,
for all~$p\in M_1$,
the vectors $C_1p,\ldots,C_\ell p$ are linearly independent in $T_pM$.
For all $p\in M$, the map $X\mapsto Xp:\Lg\to T_pM$ has kernel~$\Lg_p$ and image~$\Lg p$,
and so $\dim(\Lg p)=(\dim\Lg)-(\dim\Lg_p)=n-k=\ell$.
Then, for all~$p\in M_1$, the set $\{C_1p,\ldots,C_\ell p\}$ is a basis of $\Lg p$.

By hypothesis, $k\in\N$, so $k\ge1$.
Because $\dim\Lg_u=k\ge1$, we see that $\Lg_u\ne\{0_\Lg\}$.
Choose $Y_0\in\Lg_u\backslash\{0_\Lg\}$.
Choose $s_0\in\R\backslash\{0\}$ such that $s_0Y_0\in\Lg_*$.
Let $Y:=s_0Y_0$.
Then $Y\in(\Lg_u\backslash\{0_\Lg\})\cap\Lg_*=\Lg_u\cap\Lg_*^\times$.

For all~$p\in M_1$, we know
both that $\{C_1p,\ldots,C_\ell p\}$ is a basis of $\Lg p$
and that $Yp\in\Lg p$.
Define $a_1,\ldots,a_\ell:M_1\to\R$ by:
for all $p\in M_1$, $\displaystyle{Yp=\sum_{j=1}^\ell[a_j(p)][C_jp]}$.
Then $a_1,\ldots,a_\ell\in C^\infty(M_1,\R)$.
Since $Y\in\Lg_u$, we get $Yu=0_u$,
and so $a_1(u)=\cdots=a_\ell(u)=0$.
Define $\chi:M_1\to\Lc$ by~$\displaystyle{\chi(p)=\sum_{j=1}^\ell\,\,[a_j(p)]C_j}$.
Then $\chi(u)=0_\Lc=0_\Lg$.
For all~$p\in M_1$, we have $\displaystyle{[\chi(p)]p}=\sum_{j=1}^\ell[a_j(p)][C_jp]=Yp$.
Let the map $f_0:M_1\to\Lg$ be defined by~$f_0(p)=Y-[\chi(p)]$.
Then, for all $p\in M_1$, we have
$$[f_0(p)]p\,\,\,=\,\,\,Yp\,-\,[\chi(p)]p\,\,\,=\,\,\,Yp\,-\,Yp\,\,\,=\,\,\,0_p,$$
so $f_0(p)\in\Lg_p$.
Also, $f_0(u)=Y-0_\Lg=Y\in\Lg_*^\times$,
so $u\in f_0^*(\Lg_*^\times)$,
so $f_0^*(\Lg_*^\times)\ne\emptyset$.
Since $\Lg_*^\times$~is open in $\Lg$ and $f_0:M_1\to\Lg$ is continuous,
we see that $f_0^*(\Lg_*^\times)$ is open in~$M_1$, and,
therefore, is open in $M$ as well.
As $M$ is locally connected,
choose a nonempty connected open subset~$M_\circ$ of $M$
such that $M_\circ\subseteq f_0^*(\Lg_*^\times)$.
Then $M_\circ\subseteq \dom[f_0^*]=M_1\subseteq M'$,
so it suffices to prove that $M_\circ\cap S\ne\emptyset$.

We have $f_0(M_\circ)\subseteq f_0(f_0^*(\Lg_*^\times))\subseteq\Lg_*^\times$.
Define $f:=f_0|M_\circ:M_\circ\to\Lg_*^\times$.
Define $F:=e\circ f:M_\circ\to G_*^\times$.
For all $p\in M_\circ$,
$f(p)=f_0(p)\in\Lg_p$,
so $F(p)=(e\circ f)(p)=e(f(p))\in e(\Lg_p)\subseteq G_p$,
and so $[F(p)]p=p$.

{\it Claim 1:}
$f$ is not constant on $M_\circ$.
{\it Proof of Claim 1:}
Since $F=e\circ f$,
it suffices to show that $F$ is not constant on $M_\circ$.
Let $g\in\im[F]$ be given.
We wish to show that $F(M_\circ)\ne\{g\}$.

We have $g\in\im[F]\subseteq G_*^\times\subseteq G^\circ\backslash\{1_G\}$,
so, as the $G^\circ$-action on $M$ is fixpoint rare,
the interior in~$M$ of $\Fix_M(g)$ is empty.
So, as $M_\circ$ is a nonempty open subset of $M$,
we get $M_\circ\not\subseteq\Fix_M(g)$.
Choose $p\in M_\circ$ such that $p\notin\Fix_M(g)$.
Then $[F(p)]p=p\ne gp$, so $F(p)\ne g$.
Then $F(M_\circ)\ne\{g\}$, as desired.
{\it End of proof of Claim 1.}

Since $M_\circ$ is connected, by Claim 1, choose $q\in M_\circ$
such that the differential of $f$ does not vanish at $q$,
{\it i.e.}, such that $(df)_q\ne\bfz_{T_qM,T_{f(q)}\Lg}$.
It suffices to show that $q\in M_\circ\cap S$.
So, since $q\in M_\circ$, it suffices to~show that $q\in S$.
That is, we wish to show that $\dim\Lg'_q<k$.
Since $G'_q\subseteq G_q$, we get $\Lg'_q\subseteq\Lg_q$.
Then, as $\dim\Lg_q=k$, we need only show that $\Lg'_q\ne\Lg_q$.

Let $r:=F(q)$ and let $R:=f(q)$.
We have $r=F(q)\in G_q$, so $rq=q$.
Also, $R=f(q)\in\Lg_q$.
Also, $R=f(q)\in\im[f]\subseteq\Lg_*^\times\subseteq\Lg_*$.
Also, $e(R)=e(f(q))=(e\circ f)(q)=F(q)=r$.
Let $\phi:=(df)_q:T_qM\to T_R\Lg$ and $\Phi:=(dF)_q:T_qM\to T_rG$ and $\varepsilon:=(de)_R:T_R\Lg\to T_rG$.
Since $F=e\circ f$ and $f(q)=R$,
by the Chain Rule, $(dF)_q=(de)_R\circ(df)_q$.
That is, $\Phi=\varepsilon\circ\phi$.
As $R\in\Lg_*$ and as $e|\Lg_*:\Lg_*\to G_*$ is a~$C^\infty$~diffeomorphism,
it follows that the map $(de)_R:T_R\Lg\to T_rG$ is a vector space isomorphism.
That is, the map $\varepsilon:T_R\Lg\to T_rG$ is a vector space isomorphism.

Since $R\in\Lg_q$,
it suffices to show that $R\notin\Lg'_q$.
So, since $e(R)=r$ and $e(\Lg'_q)\subseteq G'_q$,
it suffices to show that $r\notin G'_q$.
Let $W:=T_qM$.
We wish to show that
there exists $w\in W$ such that $rw\ne w$.

Since $\phi=(df)_q\ne\bfz_{T_qM,T_{f(q)}\Lg}=\bfz_{W,T_R\Lg}$,
choose $w\in W$ such that $\phi(w)\ne0_{T_R\Lg}$.
We wish to show that $rw\ne w$.

For all $X\in\Lg$, let $\widehat{X}:=(d/dt)_{t=0}(R+tX)\in T_R\Lg$.
Then the function $X\mapsto\widehat{X}:\Lg\to T_R\Lg$ is a vector space isomorphism.
So, because we have $\phi(w)\in\im[\phi]\subseteq T_R\Lg$,
choose $B\in\Lg$ such that $\widehat{B}=\phi(w)$.
Then, by~the choice of $w$, $\widehat{B}\ne0_{T_R\Lg}$.
So, by injectivity of $X\mapsto\widehat{X}:\Lg\to T_R\Lg$, we get $B\ne0_{\Lg}$.
Let $b:=\Phi(w)$.
Then $b=\Phi(w)=(\varepsilon\circ\phi)(w)=\varepsilon(\phi(w))=\varepsilon(\widehat{B})$.

We have $\im[\chi]\subseteq\Lc$.
Then, for all $p\in M_\circ$, we see that
$$f(p)\,\,\,=\,\,\,f_0(p)\,\,\,=\,\,\,Y-[\chi(p)]\,\,\,\in\,\,\,Y\,-\,\Lc\,\,\,=\,\,\,Y\,+\,\Lc.$$
Then $\im[f]\subseteq Y+\Lc$.
Then $R=f(q)\in\im[f]\subseteq Y+\Lc$.
Also, we have $w\in W=T_qM$ and $f(q)=R$.
Then $(df)_q(w)\in T_R(Y+\Lc)$.

For all $\Ls\subseteq\Lg$, let $\widehat{\Ls}:=\{\widehat{X}\,|\,X\in\Ls\}$.
For any subspace $\Ls$ of $\Lg$,
if $R\in\Ls$, then $T_R\Ls=\widehat{\Ls}$.
So, since $R\in\Lg_q$, we get $T_R\Lg_q=\widehat{\,\Lg_q}$.
For any subspace $\Ls$ of~$\Lg$, for any $X\in\Lg$,
if $R\in X+\Ls$, then $T_R(X+\Ls)=\widehat{\Ls}$.
So, since $R\in Y+\Lc$, we get $T_R(Y+\Lc)=\widehat{\Lc}$.
Then
$$\widehat{B}\,\,\,=\,\,\,\phi(w)\,\,\,=\,\,\,(df)_q(w)\,\,\,\in\,\,\,T_R(Y+\Lc)\,\,\,=\,\,\,\widehat{\Lc}.$$
Then, by injectivity of $X\mapsto\widehat{X}:\Lg\to T_R\Lg$,
we get $B\in\Lc$.
Since $q\in M_\circ\subseteq M_1$,
it follows that $C_1q,\ldots,C_\ell q$ are linearly independent.
So, as the $\R$-span of $\{C_1,\ldots,C_\ell\}$ is $\Lc$,
the map $C\mapsto Cq:\Lc\to T_qM$ is injective.
So, since $B\in\Lc$ and $B\ne0_\Lg=0_\Lc$, it follows that $Bq\ne0_q$.
Then $B\notin\Lg_q$, so, by injectivity of $X\mapsto\widehat{X}:\Lg\to T_R\Lg$,
we get $\widehat{B}\notin\widehat{\,\Lg_q}$,
and so, by injectivity of $\varepsilon:T_R\Lg\to T_rG$,
we get $\varepsilon(\widehat{B})\notin\varepsilon(\widehat{\,\Lg_q})$.

We have $b=\Phi(w)\in\im[\Phi]\subseteq T_rG$ and $w\in W=T_qM$.
Then $bq\in[T_rG]q\subseteq T_{rq}M$
and that $rw\in r[T_qM]=T_{rq}M$.
Recall: $rq=q$.
Then $bq,rw\in T_{rq}M=T_qM=W$.
We wish to show that $w-rw\ne0_W$.

As $\varepsilon:T_R\Lg\to T_rG$ is a vector space isomorphism,
we conclude that $\dim(\varepsilon(T_R\Lg_q))=\dim(T_R\Lg_q)$.
As $R\in\Lg_q$ and $e(R)=r$ and $e(\Lg_q)\subseteq G_q$,
it follows that $(de)_R(T_R\Lg_q)\subseteq T_rG_q$.
That is, $\varepsilon(T_R\Lg_q)\subseteq T_rG_q$.
So, since
$$\dim(\varepsilon(T_R\Lg_q))=\dim(T_R\Lg_q)=\dim\Lg_q=\dim G_q=\dim(T_rG_q),$$
we get $\varepsilon(T_R\Lg_q)=T_rG_q$.
Then $b=\varepsilon(\widehat{B})\notin\varepsilon(\widehat{\,\Lg_q})=\varepsilon(T_R\Lg_q)=T_rG_q$.
As $r\in G_q$,
by \cref{cor-orbit-map-deriv} (with $p$ replaced by $q$ and $a$ by $r$),
it follows that $\ker[(d\,\overline{q})_r]=T_rG_q$.
So, since $b\notin T_rG_q$,
we get $(d\,\overline{q})_r(b)\ne0_q$.
Since $b\in T_rG$, we get $(d\,\overline{q})(b)=(d\,\overline{q})_r(b)$.
Since $W=T_qM$, we have $0_W=0_q$.
Then $bq=(d\,\overline{q})(b)=(d\,\overline{q})_r(b)\ne0_q=0_W$.

Define $a:G\times M\to M$ by $a(g,p)=gp$.
Define $F_1:M_\circ\to G\times M$ by $F_1(p)=(F(p),p)$.
Then $(dF_1)(w)=(\Phi(w),w)=(b,w)$.
Recall that, for all $p\in M_\circ$, we have $[F(p)]p=p$.
Then, for all $p\in M_\circ$, we have
$$(a\circ F_1)(p)\,\,=\,\,a(F_1(p))\,\,=\,\,a(F(p),p)\,\,=\,\,[F(p)]p\,\,=\,\,p.$$
Then $(d(a\circ F_1))(w)=w$.
By the Chain Rule,
$d(a\circ F_1)=(da)\circ(dF_1)$.
For all $g\in G$, $a(g,q)=gq=\overline{q}(g)$.
Then $(da)(b,0_q)=(d\,\overline{q})(b)=bq$.
For all~$p\in M$, $a(r,p)=rp=\overline{r}(p)$.
Then $(da)(0_r,w)=(d\,\overline{r})(w)=rw$.
Putting all this together, we get
\begin{eqnarray*}
w&=&(d(a\circ F_1))(w)\,\,=\,\,((da)\circ(dF_1))(w)\,\,=\,\,(da)((dF_1)(w))\\
&=&(da)(b,w)\,\,=\,\,[(da)(b,0_q)]+[(da)(0_r,w)]\,\,=\,\,bq+rw.
\end{eqnarray*}
Then $w-rw=bq\ne0_W$, as desired.
\end{proof}

\begin{lem}\wrlab{lem-first-order-stab-small}
Assume the $G^\circ$-action on $M$ is fixpoint rare.
Let $k\in\N$.
Assume: $\forall^\circ p\in M$, $\dim\Lg_p\le k$.
Then: $\forall^\circ p\in M$, $\dim\Lg'_p\le k-1$.
\end{lem}

\begin{proof}
For all $L\subseteq M$, let $\overline{L}$ denote the closure in $M$ of $L$.

Let $S:=\{p\in M\,|\,\dim\Lg'_p<k\}$.
Then, by upper semi-continuity of~$p\mapsto\dim\Lg'_p:M\to\N_0$,
we conclude that $S$ is open in $M$.
So, since $S=\{p\in M\,|\,\dim\Lg'_p\le k-1\}$,
it suffices to~show that $S$ is dense in $M$.

Let $U:=\{p\in M\,|\,\dim\Lg_p<k\}$.
For all $p\in M$, we have $G'_p\subseteq G_p$,
and so $\Lg'_p\subseteq\Lg_p$.
Then $U\subseteq S$, so $\overline{U}\subseteq\overline{S}$.
Let $M_1:=\{p\in M\,|\,\dim\Lg_p<k+1\}$.
By upper semi-continuity of $p\mapsto\dim\Lg_p:M\to\N_0$, we see that $M_1$~is open in $M$.
Because $M_1=\{p\in M\,|\,\dim\Lg_p\le k\}$,
by hypothesis, we know that $M_1$ contains a dense open subset of $M$.
Then $\overline{M_1}=M$.

Let $M_0:=M_1\backslash\overline{U}$.
Then $M_0$ is a $G$-invariant open subset of $M$.
Since the $G$-action on $M$ is fixpoint rare,
it follows that the $G$-action on $M_0$ is fixpoint rare as well.
For all $p\in M_0$, $\dim\Lg_p=k$.
So, by \lref{lem-decay-moving-stabs} (with $M$ replaced by~$M_0$),
$\forall^\circ p\in M_0,\dim\Lg'_p<k$.
Let $S_0:=\{p\in M_0\,|\,\dim\Lg'_p<k\}$.
Then $S_0$ contains a dense open subset of~$M_0$.
Then $M_0\subseteq\overline{S_0}$.
Since $S_0=M_0\cap S\subseteq S$, we get $\overline{S_0}\subseteq\overline{S}$.
Then $M_1\backslash\overline{U}=M_0\subseteq\overline{S_0}\subseteq\overline{S}$.
Then $M_1\subseteq\overline{U}\cup(M_1\backslash\overline{U})\subseteq\overline{S}\cup\overline{S}=\overline{S}$,
and so $\overline{M_1}\subseteq\overline{S}$.
Then $M=\overline{M_1}\subseteq\overline{S}$,
so $S$ is dense in $M$, as desired.
\end{proof}

\section{Local freeness\wrlab{sect-local-freeness}}

Here is a frame bundle analogue of Theorem 6.14 of \cite{olver:movfrmsing}:

\begin{thm}\wrlab{thm-high-stabs}
Let a Lie group $G$ act on a manifold~$M$.
Assume that the action is $C^\infty$.
Let $G^\circ$ denote the identity component of $G$.
Assume that the $G^\circ$-action on $M$ is fixpoint rare.
Let $n:=\dim G$.
Assume that $n\ge1$.
Then there exists a $G$-invariant dense open subset~$Q$ of $F^{n-1}M$
such that the $G$-action on~$Q$ is locally free.
\end{thm}

We will argue below that, at each level in the frame bundle tower,
if the generic stabilizer dimension is positive,
then, at the next level up, it will decrease by at least one.
By fixpoint rarity, it starts out less than $n$, so,
after $n-1$ transitions, it must be zero.
Details follow.

\begin{proof}
Let $\Lg:=T_{1_G}G$.
For all $p\in M$, let $G_p:=\Stab_G(p)$, $\Lg_p:=T_{1_G}G_p$.

Since $\dim G^\circ=\dim G=n\ge1$, we have $G^\circ\ne\{1_G\}$.
Choose $g_0\in G^\circ\backslash\{1_G\}$.
Let $F_0:=\Fix_X(g_0)$.
By continuity of $p\mapsto g_0p:M\to M$,
$F_0$ is closed in $M$.
The $G^\circ$-action on $M$ is fixpoint rare,
so $F_0$ has empty interior in~$M$.
Let $U_0:=M\backslash F_0$.
Then $U_0$ is a dense open subset of $M$.

{\it Claim 1:} Let $p\in U_0$. Then $\dim\Lg_p<n$.
{\it Proof of Claim 1:}
Since $\Lg_p\subseteq\Lg$ and since $n=\dim G=\dim\Lg$,
it suffices to show that $\Lg_p\ne\Lg$.
Assume that $\Lg_p=\Lg$.
We aim for a contradiction.

Since $\Lg\subseteq\Lg_p$, we get $G^\circ\subseteq G_p$.
Then $g_0\in G^\circ\subseteq G_p$, so $g_0p=p$.
Then $p\in\Fix_M(g_0)=F_0$.
Also, $p\in U_0=M\backslash F_0$, so $p\notin F_0$.
Then both $p\in F_0$ and $p\notin F_0$.
Contradiction.
{\it End of proof of Claim 1.}

For all $j\in\N_0$, let $M_j:=F^jM$ and $\pi_j:=\pi_M^j$.
The $G$-action on~$M$ induces a $G$-action on $M_j$ and this,
in turn, induces a $G$-action on $TM_j$.
For all~$j\in\N_0$, for all~$q\in M_j$,
let $G_q:=\Stab_G(q)$, let $G'_q:=\Stab'_G(q)$,
let $\Lg_q:=T_{1_G}G_q$,
and let $\Lg'_q:=T_{1_G}G'_q$.
For all $g\in G$, let $\overline{g}:M\to M$ be defined by~$\overline{g}(p)=gp$.
Let $I:=\id_M:M\to M$.
For all $p\in M$, for all~$j\in\N_0$,
let $G_p^{(j)}:=\{g\in G\,|\,\overline{g}\sim I\,\,[[j,p]]\}$,
and let $\Lg_p^{(j)}:=T_{1_G}G_p^{(j)}$.

{\it Claim 2:} Let $j\in\N_0$.
Let $q\in M_j$ and let $p:=\pi_j(q)$.
Then $\Lg_q\subseteq\Lg_p^{(j)}$.
{\it Proof of~Claim 2:}
It suffices to show that $G_q\subseteq G_p^{(j)}$.
Let $g\in G_q$ be given.
We wish to prove that $g\in G_p^{(j)}$.

Since $g\in G_q$, $gq=q$.
Then $(F^j\overline{g})(q)=gq=q=(F^jI)(q)$.
So, by~\lref{lem-stab-above-highstab-below},
$\overline{g}\sim I\,\,[[j,p]]$.
So $g\in G_p^{(j)}$.
{\it End of proof of Claim 2.}

{\it Claim 3:} Let $j\in\N_0$.
Let $q\in M_j$ and let $p:=\pi_j(q)$.
Then $\Lg_p^{(j+1)}\subseteq\Lg'_q$.
{\it Proof of~Claim 3:}
It suffices to show that $G_p^{(j+1)}\subseteq G'_q$.
Let $g\in G_p^{(j+1)}$ be given.
We wish to prove that $g\in G'_q$.

Since $g\in G_p^{(j+1)}$, we get $\overline{g}\sim I\,\,[[j+1,p]]$.
So, by $\Rightarrow$ of~\lref{lem-lose-j-derivs-from-ell} (with $i$ replaced by $1$),
$F^j\overline{g}\sim F^jI\,\,[[1,q]]$.
Then $(d(F^j\overline{g}))_q=(d(F^jI))_q$.
Then, for all $v\in T_qM_j$, we have $gv=(d(F^j\overline{g}))_q(v)=(d(F^jI))_q(v)=v$.
It follows that $g\in\Stab'_G(q)=G'_q$, as desired.
{\it End of proof of Claim 3.}

Let $m:=n-1$.
Let $Q:=\{q\in M_m\,|\,\Lg_q=\{0_\Lg\}\}$.
Then we have $Q\subseteq M_m=M_{n-1}=F^{n-1}M$.
Also, $Q$~is $G$-invariant, and, moreover, the $G$-action on $Q$ is locally free.
Also, $Q=\{q\in M_m\,|\,\dim\Lg_q<1\}$,
so, by~upper semi-continuity of~the mapping $q\mapsto\dim\Lg_q:M_m\to\N_0$,
we see that $Q$ is an open subset of~$M_m$.
It remains to show that $Q$ is dense in~$M_m$.
It suffices to show that $Q$ contains a dense open subset of $M_m$.
That is, it suffices to prove: $\forall^\circ q\in M_m$, $\Lg_q=\{0_\Lg\}$.

The preimage, under $\pi_m:M_m\to M$, of a dense open set is dense open.
It therefore suffices, by~Claim~2, to show: $\forall^\circ p\in M$, $\Lg_p^{(m)}=\{0_\Lg\}$.

For all $j\in\N_0$,
let $K_j:=\{k\in\N_0\,|\,\forall^\circ p\in M,\,\dim\Lg_p^{(j)}\le k\}$.
Then $K_0\subseteq K_1\subseteq K_2\subseteq\cdots$.
We wish to show that $0\in K_m$.
We will, in fact, show, for all integers $j\in[0,m]$, that $m-j\in K_j$.

For all $p\in M$, we have $\Lg_p^{(0)}=\Lg_p$.
So, by Claim 1, for all $p\in U_0$, we have $\dim\Lg_p^{(0)}<n$,
so $\dim\Lg_p^{(0)}\le n-1=m$.
Then, because $U_0$ is a dense open subset of $M$, we get $m\in K_0$.
Therefore, by induction, it suffices to show, for every integer $j\in[0,m-1]$, that
$${\bm(}\,m-j\in K_j\,{\bm)}\qquad\Rightarrow\qquad{\bm(}\,m-j-1\in K_{j+1}\,{\bm)}.$$
Let an integer $j\in[0,m-1]$ be given, let $k:=m-j$,
and assume that $k\in K_j$.
We wish to show that $k-1\in K_{j+1}$.

We have $j\le m-1$.
Then $k=m-j\ge1$.
Then $k\in\N$.
Since $k\in K_j$, we get: $\forall^\circ p\in M$, $\dim\Lg_p^{(j)}\le k$.
The preimage, under~$\pi_j:M_j\to M$, of~a~dense open set is dense open.
Therefore, by~Claim~2, $\forall^\circ q\in M_j$, $\dim\Lg_q\le k$.
Since $\pi_j:M_j\to M$ is open and $G$-equivariant,
and since the $G^\circ$-action on $M$ is fixpoint rare,
we see that the $G^\circ$-action on $M_j$ is fixpoint rare as well.
Then, by \lref{lem-first-order-stab-small} (with $M$ replaced by~$M_j$),
we have: $\forall^\circ q\in M_j$, $\dim\Lg'_q\le k-1$.
The image, under~$\pi_j:M_j\to M$, of a dense open set is dense open.
Therefore, by~Claim~3, $\forall^\circ p\in M$, $\dim\Lg_p^{(j+1)}\le k-1$.
Then $k-1\in K_{j+1}$, as desired.
\end{proof}

We cannot replace ``$F^{n-1}M$'' by ``$F^{n-2}M$'' in \tref{thm-high-stabs}:

\vskip.1in\noindent
{\it Example:} (Cf.~Example 4.3 in \cite{olver:movfrmsing}.)
Let $n\ge2$ be an integer.
Let $G$~be the additive Lie group~$\R^n$.
Let $M$ be the manifold $\R^2$.
Let $G$ act on~$M$ by:
$(s_1,\ldots,s_n)\,(x,y)=(x,y+s_1x+s_2x^2+\cdots+s_nx^n)$.
Since every nonzero polynomial has only finitely many roots,
we conclude, for all~$s\in G\backslash\{0_n\}$, that there exists a finite set $A\subseteq\R$
such that $\Fix_M(s)=A\times\R$.
Thus, the $G$-action on $M$ is fixpoint rare.
It is, moreover, $C^\infty$,
and, therefore, induces a $G$-action on $F^{n-2}M$.
We wish to prove: For any dense open $G$-invariant subset $Q$ of $F^{n-2}M$,
there exists $q\in Q$ such that $\Stab_G(q)$ is not discrete in $G$.
We will, in fact, show: $\forall q\in F^{n-2}M$, $\Stab_G(q)$ is not discrete in $G$.
Let $q_0\in F^{n-2}M$ be given.
We wish to prove: $\Stab_G(q_0)$ is not discrete in~$G$.

For any $\phi\in C^\infty(\R,\R)$,
for all integers $i\ge0$,
let $\phi^{(i)}:\R\to\R$ denote the $i$th derivative of $\phi$.
For all~integers $j\in[1,n]$, define $f_j:\R\to\R$ by $f_j(x)=x^j$.
Let $p_0:=\pi_M^{n-2}(q_0)$.
Then $p_0\in M=\R^2$.
Let $x_0,y_0\in\R$ satisfy $p_0=(x_0,y_0)$.
For all integers $i\in[1,n-1]$, define $T_i:\R^n\to\R$ by
$\displaystyle{T_i(s_1,\ldots,s_n)=\sum_{j=1}^ns_j[f_j^{(i-1)}(x_0)]}$;
then $T_i:\R^n\to\R$ is linear.
Define $T:\R^n\to\R^{n-1}$ by $T(s)=(T_1(s),\ldots,T_{n-1}(s))$;
then $T:\R^n\to\R^{n-1}$ is a linear transformation.
Because $\dim(\R^n)>\dim(\R^{n-1})$,
we conclude that $\dim(\ker T)>0$,
and, therefore, that $\ker T$ is not discrete in $\R^n$.
So, since $G=\R^n$, it suffices to~show that $\ker T\subseteq\Stab_G(q_0)$.
Let $s\in\ker T$ be given.
We wish to show that $sq_0=q_0$.

Define $\sigma:M\to M$ by $\sigma(p)=sp$.
Let $I:=\id_M:M\to M$.
Then $(F^{n-2}\sigma)(q_0)=sq_0$ and $(F^{n-2}I)(q_0)=q_0$,
so we wish to show: $(F^{n-2}\sigma)(q_0)=(F^{n-2}I)(q_0)$.
In other words, we wish to prove that $F^{n-2}\sigma\sim F^{n-2}I\,\,[[0,q_0]]$.
By $\Rightarrow$ of~\lref{lem-lose-j-derivs-from-ell}
(with $i$ replaced by~$0$ and $j$ by $n-2$),
it suffices to show: $\sigma\sim I\,\,[[n-2,p_0]]$.
Let $\alpha\in\N_0^2$ be given, and assume $|\alpha|\le n-2$.
We wish to show: $(\partial^\alpha\sigma)(p_0)=(\partial^\alpha I)(p_0)$.

Let $a,b\in\N_0$ satisfy $\alpha=(a,b)$.
Then $a+b=|\alpha|\le n-2$.
Then $a\in\{0,\ldots,n-2\}$.
We wish to show: $(\partial^{(a,b)}\sigma)(x_0,y_0)=(\partial^{(a,b)}I)(x_0,y_0)$.

As $s\in\ker T$, $T(s)=0_{n-1}$.
Then $(T_1(s),\ldots,T_{n-1}(s))=0_{n-1}$.
So, since $a\in\{0,\ldots,n-2\}$, we get $T_{a+1}(s)=0$.
Let $s_1,\ldots,s_n\in\R$ satisfy $s=(s_1,\ldots,s_n)$.
Let $\displaystyle{f:=\sum_{j=1}^n\,s_jf_j}$.
Then, for all~$x\in\R$, we have $f(x)=s_1x+s_2x^2+\cdots+s_nx^n$.
Therefore, for all $x,y\in\R$, we have
\begin{eqnarray*}
\sigma(x,y)&=&(s_1,\ldots,s_n)(x,y)\,\,=\,\,(x,y+[f(x)])\\
&=&(x,y)+(0,f(x))\,\,=\,\,[I(x,y)]+(0,f(x)).
\end{eqnarray*}
Then $(\partial^{(a,b)}\sigma)(x_0,y_0)=[(\partial^{(a,b)}I)(x_0,y_0)]+(0,f^{(a)}(x_0))$.
So, since
$$f^{(a)}(x_0)\quad=\quad\sum_{j=1}^n\,s_j[f_j^{(a)}(x_0)]\quad=\quad T_{a+1}(s)\quad=\quad0,$$
we get $(\partial^{(a,b)}\sigma)(x_0,y_0)=(\partial^{(a,b)}I)(x_0,y_0)$, as desired.
{\it End of example.}


\bibliography{list}

\begin{thebibliography}{Zeg95b}

\bibitem[Olver]{olver:movfrmsing}
P.~J.~Olver.
\newblock Moving frames and singularities of prolonged group actions.
\newblock Selecta Math.~{\bf6} (2000), 41--77.

\bibitem[MZ]{montzip:transfgps}
D.~Montgomery and L.~Zippin.
\newblock Topological Transformation Groups.
\newblock {\it Interscience Publishers, a division of John Wiley \& Sons},
New York, 1955.

\end{thebibliography}

\end{document}